\documentclass[11pt]{article}

\usepackage[T1]{fontenc}
\usepackage[utf8]{inputenc}
\usepackage{authblk}
\usepackage[margin=1in]{geometry}
\usepackage{amsmath,amssymb,amsthm}
\usepackage[usenames,dvipsnames,svgnames,table]{xcolor}
\usepackage{graphicx}
\usepackage{stmaryrd}
\usepackage{mathrsfs}
\usepackage{enumitem}
\usepackage{subcaption}
\usepackage{glossaries}
\usepackage{verbatim}
\usepackage{tcolorbox}

\newcommand{\R}{\mathbb{R}}
\newcommand{\Tan}{\mathrm{Tan}}
\newcommand{\B}{B}
\newcommand{\U}{U}
\newcommand{\N}{\mathbb{N}}
\newcommand{\Hd}{\mathcal{H}}

\newcommand{\reg}{\mathrm}

\DeclareMathAlphabet{\pazocal}{OMS}{zplm}{m}{n}
\newtheorem*{thm*}{Theorem}

\usepackage{mathtools}

\DeclarePairedDelimiter\floor{\lfloor}{\rfloor}

\newtheorem{thm}{Theorem}[section]
\newtheorem{lem}[thm]{Lemma}
\newtheorem{proposition}[thm]{Proposition}

\newtheorem{corollary}[thm]{Corollary}
\newtheorem{problem}[thm]{Problem}

\theoremstyle{definition}
\newtheorem{definition}[thm]{Definition}
\newtheorem*{nonNumDef}{Definition}

\theoremstyle{remark}
\newtheorem{remark}[thm]{Remark}

\usepackage{hyperref}
\usepackage{parskip}
\hypersetup{
	colorlinks=true,
	linkcolor=blue,
	filecolor=magenta,      
	urlcolor=cyan,
}

\definecolor{darkgrn}{rgb}{0, 0.75, 0}

\title{The Maximum Distance Problem and Minimum Spanning Trees}
\author[1]{Enrique Alvarado\thanks{enrique.alvarado@wsu.edu}}
\author[1]{Bala Krishnamoorthy\thanks{kbala@wsu.edu}}
\author[1]{Kevin R. Vixie \thanks{vixie@speakeasy.net}}
\affil[1]{Department of Mathematics and Statistics, Washington State University}

\graphicspath{{Figures/}}

\begin{document}
\maketitle
\begin{abstract}
  Given a compact $E\subset \R^n$ and $s > 0$, the maximum distance problem seeks 
  a compact and connected subset of $\R^n$ of smallest one dimensional Hausdorff measure whose $s$-neighborhood covers $E$.
  For $E\subset \R^2$, we prove that minimizing over minimum spanning trees that connect the centers of balls of radius $s$, which cover $E$, solves the maximum distance problem.
  The main difficulty in proving this result is overcome by the proof of Lemma~\ref{lem:main}, which states that one is able to cover the $s$-neighborhood of a Lipschitz curve $\Gamma$ in $\R^2$ with a finite number of balls of radius $s$, and connect their centers with another Lipschitz curve $\Gamma_\ast$, where $\mathcal{H}^1(\Gamma_\ast)$ is arbitrarily close to $\mathcal{H}^1(\Gamma)$.
  We also present an open source package for computational exploration of the maximum distance problem using minimum spanning trees, available at \href{https://github.com/mtdaydream/MDP_MST}{github.com/mtdaydream/MDP\_MST}.
\end{abstract}
{\small{\tableofcontents}}

\section{Introduction}\label{sec:intro}
	
There are many variants of the {\it traveling salesman problem} in $\R^2$.
The classic problem seeks the shortest connected tour through a finite collection of points $E = \{x_i\}_{i = 1}^N\subset\R^2$, where the points represent cities a salesman has to visit.
One variant of the TSP is the {\it analyst's traveling salesman problem} (ATSP) \cite{jones1990rectifiable,okikiolu1992characterization,Sc2007}, which essentially asks the same question, except for one crucial difference---the set $E$ is not restricted to finite collections of points (otherwise it reduces to the classical traveling salesman problem).
The ATSP seeks necessary and sufficient conditions for the existence of a \emph{finite continuum} $\Gamma$ containing $E$, where, by \emph{finite continuum} we mean a set which is compact, connected and has finite $\Hd^1$ measure.
Here $\Hd^1(\Gamma)$ is the one dimensional Hausdorff measure of $\Gamma$ (see Definition~\ref{def:length}).

Because for general sets in $E\subset\R^2$, it is often the case that $E$ is not contained in {\it any} finite continuum, we might consider trying to find a finite continuum, $\Gamma$, of smallest $1$-dimensional Hausdorff measure, such that the maximum distance from $\Gamma$ to any point in $E$ is at most $s > 0$.
This is the problem we focus on in the current paper.

We will see that in fact, one could just as easily have defined a finite continuum to be a compact, connected, $1$-rectifiable set of finite $\Hd^1$ measure (or even the Lipschitz image of a compact interval) by using the ideas presented by Falconer~\cite[\S3.2]{falconer-1986-geometry}, or in a slightly more precise form in Theorem~\ref{thm:davidSemmes}, which is stated and sketched by David and Semmes \cite[\S1.1]{david1993analysis}.
We discuss this aspect in more detail in Section \ref{sec:def}.
For those who are not familiar with these types of characterizations of rectifiable sets (which is a very active area of current research), we recommend they start with the excellent book by Falconer \cite{falconer-1986-geometry}.

\subsection{The Maximum Distance Problem and Steiner Trees}
\label{sec:MDP-steiner}

As stated in the introduction the minimization problem we focus on in this paper is:
\begin{align}\label{eq:maxDistMin}
  \lambda(E, s) := \min\{\mathcal{H}^1(K) : K\ \mathrm{is\ a\ finite\ continuum\ and}\ E\subset B(K, s)\},
\end{align}
where $B(K, s) := \{x \in \R^2 : \mathrm{dist}(K, x) \leq s\}$ is the closed $s$-neighborhood of $K$.
In the literature, this problem is called the {\it maximum distance problem}, or MDP in short, and we will use that name to refer to it here.
A finite continuum $\Gamma$ such that $B(\Gamma, s)\supset E$ and $\mathcal{H}^1(K)=\lambda(E, s)$ is called a {\it minimizer of $\lambda(E, s)$}, or an {\it $s$-maximum distance minimizer} of $E$.
As we will see, for compact $E\subset \R^n$ and $s > 0$, minimizers of $\lambda(E, s)$ always exist.

Note that any bounded $E\subset\R^2$ is clearly contained in the $s$-neighborhood of a finite continuum (for any $s > 0$).
Therefore, asking for sufficient and necessary conditions for the existence of such a set, in analogy to the ATSP question, is not interesting.
The existence of minimizers, i.e., finding a $\Gamma$ such that $\Hd^1(\Gamma) = \lambda(E, s)$ and $B(\Gamma, s) \supset E$, is more interesting, but is straightforward using a standard application of Go\l{}\k{a}b's Theorem.
The reader can see Falconer's book~\cite{falconer-1986-geometry} for the Go\l{}\k{a}b's Theorem in $\R^n$, or Section 4.4 of Ambrosio and Tilli's \textit{Topics on Analysis in Metric Spaces} \cite{ambrosio-2004-topics}, where the authors use these facts to get existence of geodesics in metric spaces.
We present the details of existence of minimizers in Section~\ref{sec:exist}.

Because of this difference with the ATSP, we focus on answering a different question that is motivated by the following simple heuristic, which we call the {\it cover-and-connect} heuristic.

\begin{tcolorbox}[center, width = 15cm, colback = white, coltitle = black, title = Cover-and-Connect, fonttitle=\bfseries, colbacktitle=white]
\begin{enumerate}[leftmargin = .5cm]
\item {\it Cover} $E$ with a finite number of balls of radius $s$, centered on a set of points $X$.
  \\
\vspace*{-0.08in}  
\item {\it Connect} all the centers in $X$ with a closed connected curve $\Gamma$.
\end{enumerate}
\end{tcolorbox}

In this paper, we let $\Gamma$ be either the {\it Steiner tree} $S_X$ over $X$, or the {\it minimum spanning tree} $T_X$ over $X$.
See Problem~\ref{prob:SP} and Problem~\ref{prob:mst} for related definitions.
Since $X \subset \Gamma$, the $s$-neighborhood of $\Gamma$ contains the balls, and therefore $E$.
Thus $\Gamma$ is a candidate minimizer.

In the cover-and-connect heuristic, note that since we are connecting all the points in $X$, we might as well connect them with a Steiner tree over $X$. This leads us to ask the following question that motivated our main theorem, Theorem \ref{thm:main}. 
\begin{center}
{\it How close is $\mathcal{H}^1(S_X)$ to $\lambda(E, s)$?} 
\end{center}
One can come up with many examples of $E$ where any Steiner tree $S_X$ generated over centers of balls that cover $E$ satisfies $\mathcal{H}^1(S_X) > \lambda(E, s)$.
For a useful and simple example, let $E$ equal the $s$-neighborhood of a finite line segment in $\R^2$ (see Figure~\ref{fig:LineSegmentCase}).
Although this example shows that we do not have strict equality with any $S_X$, the main result of this paper shows that there is a sequence of finite point sets $X_i = \{x^i_k\}_{k = 1}^{n_i}\subset \R^2$ such that $E\subset B(X_i, s)$ for all $i$ and
$\Hd^1(S_{X_i})\rightarrow \lambda(E, s)$ as $i\rightarrow\infty$.

In particular, for a given compact $E\subset \R^2$ and
$s > 0$, defining the {\it{$s$-spanning length of $E$}} as
\begin{align}\label{eq:spanninglength}
  \sigma(E, s) :=& \inf\{\mathcal{H}^1(S_X) : X = \{x_i\}_{i = 1}^N,\ B(X, s) \supset E\},
\end{align}
we establish the following main result.
{
\renewcommand{\thethm}{\ref{thm:main}}
  \begin{thm} 
    Let $E\subset \R^2$ be compact and let $s > 0$. Then
    \begin{align*}
      \sigma(E, s) = \lambda(E, s).
    \end{align*}
  \end{thm}
}

\begin{remark}\label{rem:main}
  In general, Steiner trees $S_X$ over $X$ may introduce a new collection of branching points $Y$ (often called {\it Steiner points} in the literature).
  If we then consider the new collection of points $X' = X \cup Y$, it is a fact that the minimum spanning tree $T_{X'} = S_X$.
  Therefore, the definition of $\sigma(E, s)$ is unchanged when replacing $S_X$ in Equation~(\ref{eq:spanninglength}) with a minimum spanning tree $T_X$.
\end{remark}

By the above remark, we prove the following main Corollary.
This is an important clarifying remark due to the fact that Steiner trees are difficult to compute, but minimum spanning trees can be computed in polynomial time.
{
\renewcommand{\thethm}{\ref{cor:main}}
  \begin{corollary} 
    Let $E\subset \R^2$ be compact and let $s > 0$. Define the analogous 
    \begin{align*}
      \sigma'(E, s) :=& \inf\{\mathcal{H}^1(T_X) : X = \{x_i\}_{i = 1}^N,\ B(X, s) \supset E\}
    \end{align*}
    where we are taking minimum spanning trees $T_X$ over $X$ instead of Steiner trees over $X$.
    Then 
    \begin{align*}
      \sigma'(E, s) = \lambda(E, s).
    \end{align*}
  \end{corollary}
}

The proof of Theorem~\ref{thm:main} will follow from Lemma~\ref{lem:main}, which constitutes the heart of our paper.
Intuitively, this lemma says that given any $\epsilon >0$, the $s$-neighborhood of any Lipschitz curve $\Gamma$ is contained in a finite number of balls of radius $s$, whose centers are connected by another finite continuum $\Gamma_\ast$ such that $\mathcal{H}^1(\Gamma_\ast)$ is within $\epsilon$ of $\mathcal{H}^1(\Gamma)$.
We present the precise statement of this Lemma below.
{
\renewcommand{\thethm}{\ref{lem:main}}
\begin{lem}
  Let $s > 0$ and $\Gamma\subset \R^2$ be a Lipschitz curve of positive length.
  Then, given $\epsilon > 0$, there exist a finite point set $X := \{x_i\}_{i = 1}^N\subset \R^2$ and a Lipschitz curve $\Gamma_\ast$ that contains $X$ such that 
  \begin{align*}
    B(X, s) \supset B(\Gamma, s) \quad \mathit{and}\quad \mathcal{H}^1(\Gamma) \leq \mathcal{H}^1(\Gamma_\ast) \leq \mathcal{H}^1(\Gamma) + \epsilon. 
  \end{align*}
\end{lem}
}

We will now briefly outline previous work on the maximum distance problem and closely related problems, such as the average distance problem, the constrained average distance problem, and its $L^p$ variants.

\subsection{Previous Work}
\label{sec:previous}

In the mathematical literature, the maximum distance problem (MDP) evolved from a different starting point than ours.
The problem was first introduced by Buttazzo, Oudet, and Stepanov~\cite{buttazzo2002optimal} when they studied optimal urban transportation networks in cities.
In their case, optimality meant minimizing the {\it average distance} between the population in the city and the transportation network itself.
More precisely, the city population was modeled as a measure $\mu$ on $\R^2$, and transportation networks were modeled as connected, compact sets $\Sigma$ with $\mathcal{H}^1(\Sigma) \leq l$, for some fixed constant
$l > 0$.
The objective was to minimize the average distance
\[\int_{\R^2} \mathrm{dist}(x,\Sigma)\,d\mu (x)\]
over all connected compact sets $\Sigma$ such that $\mathcal{H}^1(\Sigma) \leq l$.
One can think of this problem as the $L^1$ version of the $L^\infty$ ``dual'' maximum distance problem, where instead of minimizing $\mathcal{H}^1(\gamma)$ over all closed connected $\gamma$ such that $E\subset B(\gamma, s)$ for a fixed $s$, we minimize $s > 0$ over all closed connected $\gamma$ such that $E\subset B(\gamma,s)$ and $\Hd^1(\gamma) \leq l$ for fixed $l$.
This is the $L^\infty$ version in the sense that, at least in the case of well behaved measures $\mu$, solving the problem yields 
\[  ||\mathrm{dist}(x, \Sigma)||^\infty_\mu :=  \inf \{r > 0 :  \mu\{ x : \mathrm{dist}(x, \Sigma) > r\} = 0\}.\]
Noting that
\[ ||\mathrm{dist}(x,\Sigma)||^\infty_\mu =   \lim_{k\rightarrow\infty} \left(\int_{\R^2} [\mathrm{dist}(x,\Sigma)]^k \,d\mu (x)\right)^{\frac{1}{k}},\]
we see that 
 \[  \min_{\{\Sigma : \mathcal{H}^1(\Sigma) \leq l\}} ||\mathrm{dist}(x,\Sigma)||^\infty_\mu = \min_{\{\Sigma : \mathcal{H}^1(\Sigma) \leq l\}}  \lim_{k\to\infty} \left(\int_{\R^2} [\mathrm{dist}(x, \Sigma)]^k \, d\mu (x)\right)^{\frac{1}{k}}, \]
and the connection to the above $L^1$ version becomes more apparent.

Paolini and Stepanov \cite{paolini2004qualitative} studied both the maximum distance problem and its dual, and were able to show that minimizers of the of the maximum distance problem and its dual are in fact equivalent in $\R^n$.\footnote{The terminology that we use when naming the "maximum distance problem" and it's ``dual'' are reversed in~\cite{paolini2004qualitative}.
What we call the {\it maximum distance problem} is termed by them as the {\it dual to the maximum distance problem}, and {\it vice versa}.
We chose to give it the same name due to the equivalent nature of the two problems.}
	
These papers began a large line of work on the average distance problem and on related problems such as the one studied here.
For an overview of the average distance problem, see the wonderful survey of Lemanent~\cite{Lem11}, and references therein.
	
Along this line of work, Teplitskaya \cite{teplitskaya2018regularity} recently announced an enlightening regularity result proven in~\cite{teplitskaya2019regularity}.
The result states that minimizers of the maximum distance problem consist of a finite number of curves which have one sided tangent lines at each point.
Teplitskaya also shows that the angles between these tangent lines are greater than or equal to $2\pi/3$.



\begin{remark}
  {\rm 
    The authors thank the anonymous reviewers of a previously submitted version of this paper for pointing us to Theorem 3.7 of Miranda Jr., Paolini, and Stepanov \cite{miranda2006one}.
    This theorem is in fact equivalent to Lemma~\ref{lem:main} in our paper, and in addition, the techniques they used in their proof are similar to ours. 
    We describe the differences in Remark~\ref{rem:MPS}.
  }
\end{remark}

\subsection{Outline of the Proofs}

As a means of illuminating the path to the proof of Lemma~\ref{lem:main}, and hence Theorem~\ref{thm:main}, we show that $\sigma(E, s) = \lambda(E, s)$ for two simpler cases.
It is our hope that in doing so, a non-expert will be able to get a better instinctive feel for the types of arguments used in the proofs of Lemma~\ref{lem:main} and Theorem~\ref{thm:main}.
In Lemma~\ref{lem:linesegment} we assume that $E$ is the $s$-neighborhood of a line segment, and then in Proposition~\ref{prop:lambdaandsigma}, we assume that the $s$-maximum distance minimizer of $E$ is a $C^1$ curve, rather than merely a finite continuum as we do so in Theorem~\ref{thm:main}.
	
The approach we take to prove Theorem~\ref{thm:main}, Lemma~\ref{lem:linesegment}, and Proposition~\ref{prop:lambdaandsigma} is to show the existence of Steiner trees $\{S_n\}$ such that $\mathcal{H}^1(S_n) \to \lambda(E, s)$ as $n \to \infty$.
Of course, as our definition of $\sigma(E, s)$ requires, each $S_n$ will be taken over a finite collection of points $X_n$ such that $B(X_n, s) \supset E$.
We meet this requirement by explicitly constructing $X_n$ and a curve $\Gamma^n_\ast$ that connects all the points in $X_n$.
Since by definition $\mathcal{H}^1(S_n) \leq \mathcal{H}^1(\Gamma^n_\ast)$, and since $\lambda(E, s) \leq \mathcal{H}^1(S_n)$, it suffices to show that $\mathcal{H}^1(\Gamma^n_\ast) \to \lambda(E, s)$ as $n\to +\infty$.
	
The key technique for proving Lemma~\ref{lem:main} is revealed in the simple case where $E$ itself is the $s$-neighborhood of a line segment $L$.
First notice that the $s$-maximum distance minimizer for $E$ is the line segment $L$.
If we want a Steiner tree $S_n$ over $X_n = \{x_i\}_{i = 1}^n$ to equal $L$, then $X_n$ must contain the endpoints of $L$, and $X_n$ must also be contained in $L$.
However, since $X_n$ only contains a finite number of points, $\B(X_n, s)$ cannot contain $E$ (see the left picture of Figure~\ref{fig:LineSegmentCase}).
\begin{figure}[htp!]
  \begin{minipage}{.5\textwidth}
    \centering \includegraphics[width=3.1in]{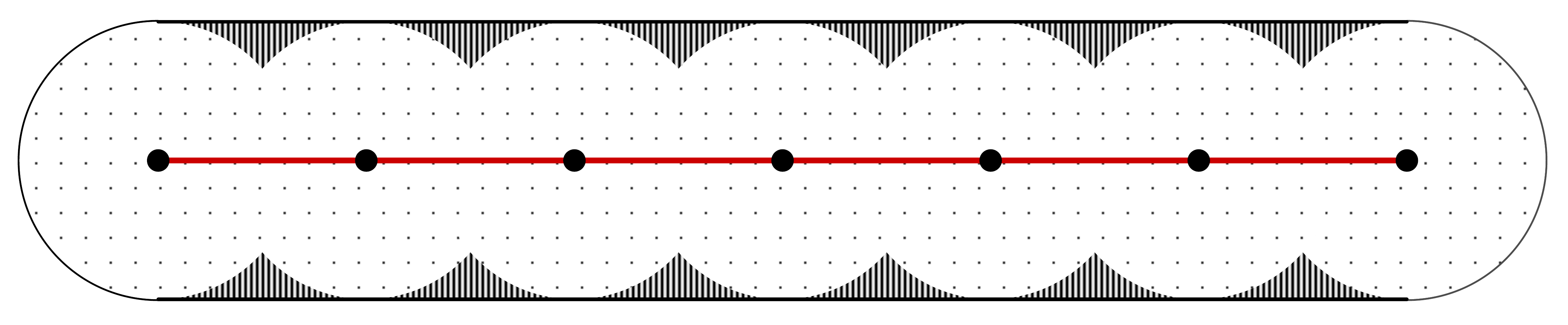}
  \end{minipage}%
  \begin{minipage}{.5\textwidth}
    \centering \includegraphics[width=3.1in]{LineSegmentCase3.pdf}
  \end{minipage}
  \caption{  \label{fig:LineSegmentCase}
    On the left, the $s$-neighborhood of the line segment is not covered by the $s$-neighborhood of a finite number of points lying on the line segment.
    On the right, we extend our previous points outwards {\it just} enough to cover the $s$-neighborhood of the line segment.
    The length of the new $1$-rectifiable set connecting the new points is equal to the union of the red and green line segments, and is not much larger than the length of the original red line segment.
    Note that although we we do not need to extend the endpoints, it will be necessary to do so in the general case. 
  }
\end{figure}

Nonetheless, as is depicted on the right picture of Figure~\ref{fig:LineSegmentCase}, we may ``extend'' each point in $X_n$ up and down (and also to the sides for the endpoints of $L$) by a small amount $\delta_n$ so that the $s$-neighborhood of these extended points $X'_n$ contains $E$.
For any large enough $n$, we have that $\mathcal{H}^1(S'_n)$ for the Steiner tree $S'_n$ over $X'_n$ is bounded above by $\mathcal{H}^1(\Gamma^n_\ast)$, where $\Gamma^n_\ast = L\cup P^n$, where $P^n$ are the short $2n + 2$ line segments, each of which is of length $\delta_n$ (see right picture of Figure~\ref{fig:LineSegmentCase}).
Since $\Gamma^n_\ast$ connects all points in $X'_n$ and since
\begin{align*}
  \mathcal{H}^1(\Gamma^n_\ast) = \mathcal{H}^1(L) + (2n + 2)\delta_n,
\end{align*}  
showing that $\mathcal{H}^1(\Gamma^n_\ast) \to \mathcal{H}^1(L) = \lambda(E, s)$ as $n\to\infty$.
Hence proving that
\begin{align}\label{eq:assymptotic}
        n\delta_n \to 0\ \mathrm{as}\ n\to \infty,
\end{align}  
would show that $\mathcal{H}^1(S'_n) \to \lambda(E, s)$.
In essence, this result (Equation \ref{eq:assymptotic}) is true due to the fact that $x/\sqrt{x} \to 0$ as $x\to 0$.
We will explore this result in greater detail in the proof of Lemma~\ref{lem:linesegment}.
    
Extending points out in similar ways is also crucial for the proofs of Theorem~\ref{thm:main} and Proposition~\ref{prop:lambdaandsigma}.
However, more care has to be taken in these more complicated cases.
In the $C^1$ case, we partition our minimizer into a finite number of pieces where each piece is contained in some uniformly thin tube.
Because of this set up, we must not only extend our points out by $\delta_n$, but also by the width of our tubes.
In the case where the $s$-neighborhood minimizer $\Gamma$ is a finite continuum, even more care must be taken.
Using a classical result of Geometric Measure Theory which states that finite continua are characterized by Lipschitz curves, we prove the theorem for Lipschitz curves; that is, images of Lipschitz functions $\gamma: I\to \R^2$ for some compact interval $I\subset \R$. 
The main difficulty for the case of a Lipschitz curve $\Gamma = \gamma(I)$ is overcome in Lemma~\ref{lem:main}.
Since we lose the uniform thinness of our tubes and differentiability at {\it all} points, we partition $\Gamma$ into a good part and a bad part.
The good part of $\Gamma$ is around points of differentiability of $\gamma$, allowing us to construct the portions of $\Gamma_\ast$ around these points as in the $C^1$ case.
Since this set of differentiable points $G\subset I$ of $\gamma$ has full measure, we may pick a compact subset $K\subset G$ such that $\mathcal{L}^1(I \setminus K) < \xi$ (for $\xi > 0$ small). This tells us that the bad portions of $\Gamma$ are small in measure and allows us to be more liberal with our construction of $\Gamma_\ast$ around these bad portions.

\section{Preliminaries}
We collect in Table \ref{tab:notation} important notation used in the paper.
We formally define the main problems we study in Section \ref{sec:probStatements}, followed by standard definitions and classical results in Section \ref{sec:def}.
\begin{table}[htp!]
  \centering
  \caption{\label{tab:notation} Notation used in the paper, and their explanations.}
  \begin{tabular}{ll}
    \hline
    Notation & definition/interpretation \\
    \hline
    $\mathbf{cl}(A)$ & closure of subset $A$ of $\R^n$\\
    $U(x, r)$ & open ball of radius $r$ centered at $x \in \R^n$\\
    $U(A, r)$ & open $r$-neighborhood of $A\subset \R^n$ when $r > 0$\\
    $B(x, r)$ & closed ball of radius $r$ centered at $x \in \R^n$\\
    $B(A, r)$ & closed $r$-neighborhood of $A\subset \R^n$ when $r > 0$\\
    $\reg{card}(A)$ & cardinality of subset $A$ of $X$\\
    $\mathcal{H}^d$ & $d$-dimensional Hausdorff measure\\
    $\mathcal{L}^n$ & $n$-dimensional Lebesgue measure \\
    $\lambda(A, s)$ & $s$-maximum distance length of $A$ \\
    $\Lambda(A, s)$ & the set of all $s$-maximum distance minimizers of $A$\\
    $\sigma(A, s)$ & $s$-spanning length of $A$\\
    $\mathrm{Tan}(S, a)$ & tangent cone of $S$ at $a$\\
    $S(V; r, t)$ & Closed asymmetric strips perpendicular to subspace $V$\\
    $V_\sharp$ & the orthogonal projection from $\R^n$ to subspace $V$ of $\R^n$\\
    $V^\perp$ & the perpendicular subspace of $V$ for subspace $V$ of $\R^n$\\
    $S_X$ & A Steiner tree over a finite point set $X\subset \R^n$\\
    $T_X$ & A minimum spanning tree over a finite point set $X\subset \R^n$ \\
    \hline
  \end{tabular}
\end{table}

\subsection{Problem statements}
\label{sec:probStatements}

\begin{problem}[Maximum Distance Problem, MDP]\label{prob:FTSP}
	Given a compact $E\subset \R^2$ and $s > 0$, compute
	\begin{align*}
	\lambda(E, s) := \inf \{\mathcal{H}^1(K) : K\ \reg{is\ closed,\ connected\ and}\ E \subset B(K, s)\}.
	\end{align*}
	We call the number $\lambda(E, s)$ the {\it{$s$-maximum distance length of $E$}}, and a closed and connected $\Gamma \subset \R^2$ such that $B(\Gamma, s) \supset E$ and $\mathcal{H}^1(\Gamma) = \lambda(E, s)$ an $s$-maximum distance minimizer of $E$, or if it is clear from context, a minimizer of $E$, or simply, a minimizer.
\end{problem}

\begin{problem}[Steiner Problem]\label{prob:SP}
  Given a finite set of points $X = \{x_i\}_{i = 1}^N$ in $\R^n$,
  compute
  \begin{align}\label{eq:SPlength}
    \inf\{\mathcal{H}^1(S) : S\ \mathrm{is\ closed,\ connected\ and}\ X\subset S\}.
  \end{align}
  The minimizers of~(\ref{eq:SPlength}), which can be shown to exist,
  are known as {\bfseries Steiner trees} over $X$ and are denoted $S_X$.
\end{problem}

\begin{problem}[Minimum Spanning Tree Problem, MST]\label{prob:mst}
Given a finite set of points $X = \{x_i\}_{i = 1}^N \subset \R^n$, and the collection of closed line segments $\mathscr{L} = \{[x, y] :  x, y \in X\}$, compute 
\begin{align}\label{eq:mst}
\min \left\{\sum_{l \in \mathscr{E}} \mathcal{H}^1(l) : \mathscr{E}\subset \mathscr{L},\ \bigcup_{\mathscr{E}} l\ \mathrm{is\ connected}\ \mathrm{and\ contains}\ X\right\}.
\end{align}
A union $T_X := \cup_{\mathscr{E_\ast}} l$, where $\mathscr{E}_\ast$ is a minimizer of~(\ref{eq:mst}) is called a {\bfseries minimum spanning tree} over $X$.
\end{problem}

\begin{remark}
  Note that even though we are not minimizing over trees in Problems \ref{prob:SP} and \ref{prob:mst}, we automatically get minimizers that are trees.
  This result follows from the observation that any possible solution with loops can always be pruned to remove loops and get a strictly shorter connected set.
\end{remark}

\begin{remark}
  Computing~(\ref{eq:mst}) in problem~(\ref{prob:mst}) can be done in polynomial time, by solving the minimum spanning tree problem for a corresponding weighted, complete graph, $K = (V, E, w)$.
  To construct $K$, let vertex $v_i \in V = \{v_i\}_{i = 1}^N$ correspond to $x_i$, and give each edge $(i, j) \in E$ the weight $w_{ij}$ that is equal to the Euclidean distance between $x_i$ and $x_j$ in $\R^n$.
\end{remark}

\subsection{Definitions and classical theorems}
\label{sec:def}

We will start with some standard definitions found in geometric measure theory literature.
In Remark~\ref{rem:equivalence}, we emphasize an important fact---that the case of characterizing finite continua in $\R^n$ plays a very special role in geometric measure theory.
This is, in part, due to the strong nature that connectivity has on sets of finite one dimensional Hausdorff measure.

\begin{definition}\label{def:neighborhood}
  For $x \in \R^n$ and $r > 0$, we let $B(x, r)$ and $U(x, r)$ denote the {\bfseries closed $r$-ball} and {\bfseries open $r$-ball} of radius $r$ centered at $x$, respectively. Similarly, for any $A\subset \R^n$ we denote the {\bfseries closed $r$-neighborhood} of $A$ as $B(A, r)$, and the {\bfseries open $r$-neighborhood} of $A$ as $U(A, r)$, and define them to be
\begin{align}\label{def:closedNeigh}
  B(A, r) := \{x \in \R^n : \mathrm{dist}(x, A) \leq r\}\quad \mathrm{and}\quad U(A, r) := \{x \in \R^n : \mathrm{dist}(x, A) < r\}.
\end{align} 
Here, $\mathrm{dist}(x, A) := \inf\{|x - y| : y \in A\}$ where $|\cdot |$ denotes the standard Euclidean distance in $\R^n$. 
\end{definition}

\begin{definition}\label{def:length}
  A {\bfseries finite continuum} $\Gamma \subset\R^n$ is a compact, connected set whose $1$-dimensional Hausdorff measure $\Hd^1(\Gamma)$ is  finite.
\end{definition}

\begin{definition}\label{def:rectifiable-set}
  A {\bfseries $\boldsymbol{1}$-rectifiable set} $\Gamma \subset\R^n$ is any set with finite $\Hd^1$ measure contained in the union of a countable collection of images of Lipschitz functions $\gamma_i : \R\to\R^n$ and a set with $\Hd^1$-measure $0$:
\[  \Hd^1(\Gamma \setminus\bigcup_{i=1}^{\infty} \gamma_i(\R)) = 0\quad \text{ and }\quad \Hd^1(\Gamma) < +\infty. \]
\end{definition}

\begin{definition}
  A subset $\Gamma\subset \R^n$ is a {\bfseries Lipschitz curve} if it is the image of some Lipschitz function $\gamma : [a, b] \to \R^n$ for $-\infty < a < b < +\infty$.
  The {\bfseries length} of $\gamma$ is defined to be 
  \begin{align*}
    \mathrm{length}(\gamma) := \sup \sum_{i = 1}^m|\gamma(t_i) - \gamma(t_{i - 1})|
  \end{align*} 
  where the supremum is taken over all disections $a = t_0 \leq t_i \leq ... \leq t_m = b$ of $[a, b]$.
\end{definition}

\begin{remark}\label{rem:equivalence}
  There are several equivalent definitions of the family of subsets of $\R^n$ that comprise finite continua. We list three of them to show the intimate connections between these definitions:
  \begin{enumerate}
    \item $\{\Gamma\subset \R^n\ |\ \Gamma\ \text{is compact, connected, and}\ \Hd^1(\Gamma) < +\infty\}$ \hfill (Finite Continua)
    \item $\{\Gamma\subset \R^n\ |\ \Gamma\ \text{is compact, connected, 1-rectifiable}\}$ \hfill (Rectifiable Continua)
    \item $\{\Gamma\subset \R^n\ |\ \Gamma= \gamma([0,L])\ \text{for some Lipschitz map}\ \gamma:[0, L]\to\R^n\}$ \hfill (Lipschitz Curves)
  \end{enumerate}

  The equivalence follows from a classic geometric measure theory result which states that any compact, connected set $\Gamma$ with $\mathcal{H}^1(\Gamma) < +\infty$ is in fact, $1$-rectifiable, and a slightly more refined result, Theorem~\ref{thm:davidSemmes}, stated next.
  This theorem tells us that $\Gamma$ is the Lipschitz image of an interval $[0,L]$ such that $L < C \Hd^1(\Gamma)$ for a $C$ that does not depend on the set $\Gamma$.
  The proof of Theorem~\ref{thm:davidSemmes} is sketched in the book by David and Semmes~\cite[\S1.1]{david1993analysis}.
  Note also that the theorem implies that $\gamma$ is parameterized by arc-length, and therefore in the statement of the theorem, $L = \mathrm{length}(\gamma)$.
\end{remark}

\begin{thm}[Theorem 1.8 in \cite{david1993analysis}]\label{thm:davidSemmes}
  There is a constant $C = C(n)$ such that whenever $\Gamma \subset \R^n$ is compact, connected such that $\mathcal{H}^1(\Gamma) < +\infty$, there is a positive number $L$ and a Lipschitz function $\gamma : [0, L] \to \R^n$ such that $\Gamma = \gamma([0, L])$, $\mathcal{H}^1(\Gamma) \leq L \leq C\mathcal{H}^1(\Gamma)$, and $|\gamma'(x)| = 1$ almost everywhere on $[0, L]$.
\end{thm}



\begin{definition}\label{def:totallybounded}
  Given $\epsilon > 0$, we say that $X\subset A\subset \R^n$ is an
  {\bfseries $\epsilon$-net} for $A$ if $A\subset B(X,\epsilon)$. If $X$ is finite, we say that $X$ is a finite $\epsilon$-net.
\end{definition}

\begin{definition} \label{del:orthcplmt}
  Let $V$ be a $k$-dimensional linear plane in $\R^n$.
  We denote by $V^\perp$ the {\bfseries orthogonal complement} of $V$, and $V_\sharp : \R^n \to V$ the {\bfseries orthogonal projection} onto $V$.
  For $\alpha > 0$, we define the {\bfseries cone} of slope $\alpha$ with respect to $V$ to be 
  \begin{align*}
    C(V, \alpha) := \{x \in \R^n : |V_\sharp(x)| \leq \alpha |V^\perp_\sharp(x)|\}.
  \end{align*}
  For every $x \in \R^n$ we denote by $C(x, V, \alpha)$ the set $x + C(V, \alpha)$.
\end{definition}

In the special case where $V$ is a $1$-dimensional linear plane of $\R^2$ with a prescribed positive direction, in Lemma~\ref{lem:main} we will be intersecting the cone with the {\bfseries asymmetric closed strip}
\begin{align*}
  S(V; [a, b]) := \{x \in \R^n :  a \leq V_\sharp(x) \leq b\}.
\end{align*}
For $x \in \R^2$, we denote by $C(x, V, \alpha; [a, b])$ the set $x + [C(V, \alpha)\cap S(V; [a, b])]$.

\begin{definition}[\S3.1.21 in~\cite{federer-1969-1}]
  Whenever $S\subset \R^n$ and $a \in \R^n$, we define the {\bfseries tangent cone} of $S$ at $a$, denoted
  \begin{align*}
    \Tan(S, a),
  \end{align*}
  as the set of all $v \in \R^n$ such that for every $\epsilon > 0$, there exists 
  \begin{align*}
    x \in S,\ 0 < r\in \R\ \mathrm{with}\ |x - a| < \epsilon,\ |r(x - a) - v| < \epsilon;
  \end{align*}
  such vectors $v$ are called {\bfseries tangent vectors} of $S$ at $a$. 
\end{definition}

\subsection{Existence of Minimizers}
\label{sec:exist}	

Using compactness results for non-empty compact subsets of $\R^n$ in a bounded portion $B$ of $\R^n$ (Blaschke selection theorem) and lower-semicontinuity of $\mathcal{H}^1$ under Hausdorff convergence (Go\l{}\k{a}b's theorem), we show existence of minimizers of $\lambda(E, s)$ for compact $E\subset \R^n$ and $s > 0$.
One can find proofs of the above theorems in~\cite[\S3.2]{falconer-1986-geometry} for the case of $\R^n$, or in~\cite[\S4.4]{ambrosio-2004-topics} for general compact metric spaces.
For every $A, B\subset \R^n$ we define the {\bfseries Hausdorff distance} between $A$ and $B$ to be
\begin{align*}
  d_H(A, B) := \inf\{r \in [0, +\infty] : B\subset B(A, r), \mathrm{and}\ A\subset B(B, r)\}.
\end{align*}

\begin{thm}[Existence]\label{thm:existence}
  For a compact subset $E$ of $\R^n$ and $s > 0$, minimizers of $\lambda(E, s)$ exist.
  That is, there exists a compact and connected $\Gamma$ such that $B(\Gamma, s) \supset E$ and $\mathcal{H}^1(\Gamma) = \lambda(E, s)$. 
\end{thm}
\begin{proof}
  Let $\{K_i\}_{i = 1}^\infty$ be a minimizing sequence; that is, for any $j = 1, 2, ...$, $K_j$ is closed, connected,  $\B(K_j, s) \supset E$, and $\lim_{i \to \infty} \mathcal{H}^1(K_i) = \lambda(E, s)$.
  Since we assume $E$ is compact, $E$ lies inside a large enough ball $\B(0, R - 2s)$ for some $R > 0$.
  We may then assume that each $K_j$ in our sequence is a subset of $\B(0, R)$, since if it were not, then projecting $K_j$ radially onto $\partial \B(0, R)$ would decrease the $\mathcal{H}^1$-measure of $K$ and $\B(K_j, s)$ would still contain $E$.
  Hence, by the Blaschke selection theorem~\cite[\S3.4]{falconer-1986-geometry}, there exists a subsequence $\{K_{i_j}\}_{j = 1}^\infty$ and a compact set $\Gamma \subset \R^n$ such that $K_{i_j}$ converges to $\Gamma$ under the Hausdorff metric as $j \to \infty$.
  Therefore, since each $K_{i_j}$ is connected, by Go\l{}\k{a}b's Theorem~\cite[\S3.2]{falconer-1986-geometry}, we have that
  \begin{align*}
    \mathcal{H}^1(\Gamma) \leq \lim_{j \to \infty}\mathcal{H}^1(K_{i_j})
  \end{align*}
  and that $\Gamma$ is connected.
  To conclude, since $K_{i_j}$ converges under the Hausdorff metric to $\Gamma$, we also have that $\B(K_{i_j}, s)$ converges to $\B(\Gamma, s)$ under the Hausdorff metric, and hence the closed set $\B(\Gamma, s)$ also contains $E$.
  Therefore $\Gamma$ is closed, connected,  $\B(\Gamma, s) \supset E$ and  $\mathcal{H}^1(\Gamma) = \lambda(E, s)$; meaning that  $\Gamma$ is a minimizer of $\lambda(E, s)$.
\end{proof}

\section{Minimizing over Minimum Spanning Trees solves the Maximum Distance Problem}\label{sec:FTSPandSteiner}

In this section, we will prove our main results: Lemma~\ref{lem:main} and Theorem~\ref{thm:main}.
We believe the intricacies that come from working with Lipschitz curves can cloud the key instincts underlying the proof, so we prove the main result for (1) line segments and then (2) for $C^1$ curves, before moving on to (3) the main theorem that obtains the same result for finite continua.

Before we treat these three cases in detail, we establish some weaker results which are easier to get due to the fact that we allow ourselves wiggle room in the distance $s$, i.e., we look at $s+\epsilon$ neighborhoods of $\Gamma$.

\subsection{($s+\epsilon$)-Neighborhoods of Steiner Trees}
\label{sec:loose-steiner}

Although we show that $\lambda(E, s) = \sigma(E, s)$ for the cases when minimizers are Lipschitz curves, we have a weaker relationship with $\lambda(E, \epsilon)$ in Proposition \ref{prop:lambdaandsigma}, where the following Lemma \ref{lem:lambdaboundssigma} becomes crucial.
	
\begin{lem}\label{lem:lambdaboundssigma}
  Let $E\subset \R^n$ be compact. If $0 < s < t$ then $\lambda(E, s) \geq \sigma(E, t)$.
\end{lem}
\begin{proof}
  First, let $\epsilon = s$ and let $\delta > 0 $ such that $\epsilon + \delta = t$; we will instead show that  $\lambda(E, \epsilon) \geq \sigma(E, \epsilon + \delta)$.
  By Theorem~\ref{thm:existence}, there exists a minimizer $\Gamma$ of $\lambda(E, \epsilon)$ that is compact, connected, and $\mathcal{H}^1(\Gamma) < +\infty$.
  Since $\Gamma$ is compact, there exists a finite $\delta$-net, $X\subset \Gamma$ of $\Gamma$.
  Recall, this  means that for any $a \in \Gamma$ there exists $b \in X$  such that $|a - b| < \delta$.
		
  Now, if we are able to show that for any $x \in \B(\Gamma, \epsilon)$ there exists a $y \in X$ such  that $|x - y| < \epsilon + \delta$, we would guarantee that  $\B(X, \epsilon + \delta) \supset \B(\Gamma, \epsilon)$.
  And since $\Gamma$ is a minimizer of $\lambda(E, \epsilon)$, $\B(\Gamma, \epsilon) \supset E$, we would then be able to say that $\B(X, \epsilon + \delta) \supset E$.
  Thus by picking a Steiner tree $S_X$ over $X$, since $X$ was originally picked to be contained in $\Gamma$, we would have that $\mathcal{H}^1(S_X) \leq \mathcal{H}^1(\Gamma)$ and therefore $\sigma(E, \epsilon + \delta) \leq \mathcal{H}^1(S_X) \leq \mathcal{H}^1(\Gamma) = \lambda(E, \epsilon)$.
		
  Let us now show this is indeed the case.
  Let $x \in \B(\Gamma, \epsilon)$.
  Since $\Gamma$ is closed, there exists a $z \in \Gamma$ such that $|x - z| \leq \epsilon$.
  Since $X$ is a $\delta$-net over $\Gamma$, we know there exists $y \in X$ such that $|y - z| < \delta$.
  Therefore by the triangle inequality, $|x - y| \leq |x - z| + |z - y| < \epsilon + \delta$.
\end{proof}

\begin{proposition}\label{prop:lambdaandsigma}
  Let $E\subset \R^2$ be compact and consider a positive sequence  $\delta_i \to 0$ as $i \to \infty$.
  There exists a sequence of finite point sets $X_i$ and Steiner trees $S_{X_i}$ such that
  \begin{align*}
    &E\subset B(X_i, \epsilon + \delta_i)\\
    &S_i \overset{H}{\to} S^* \text{ and }\\
    &\lim_{i\to \infty}\mathcal{H}^1(S_{X_i}) = \mathcal{H}^1(S^*) = \lambda(E, \epsilon).
  \end{align*} 
\end{proposition}
	
\begin{proof} We break the argument into steps:
  \begin{enumerate}
    \item Define $\sigma(s) := \sigma(E, s)$ and $\lambda(s) := \lambda(E, s)$.
    \item Because $\delta_i > 0$ for all $i$, Lemma \ref{lem:lambdaboundssigma} implies that $\sigma(\epsilon + \delta_i) \leq \lambda(\epsilon)$ for all $i$.
    \item We can therefore find $X_i$ such that
      \begin{enumerate}
        \item  $ \Hd^1(S_{X_i}) \leq \lambda(\epsilon) + \delta_i$\and{, and}
        \item  $E\subset B(S_{X_i},\epsilon+\delta_i)$.
      \end{enumerate}
    \item Recalling the argument in the Theorem~\ref{thm:existence}, since each $S_{X_i}$ is closed and connected in a compact metric space $(\B(0, R), ||\cdot||_2)$ there exists a closed and connected set $S^*$ and a subsequence such that $S_{X_{i(k)}} \overset{H}{\to} S^*$.
      Recall that $H$ denotes the Hausdorff metric here.
    \item \label{step:boundS*} By the lower semicontinuity of $\Hd^1$, we can conclude that
      \begin{align*}
        \mathcal{H}^1(S^*) & \leq \liminf_{k\to \infty}\mathcal{H}^1(S_{X_{i(k)}})\\
                           &  \leq \lambda(\epsilon).
      \end{align*}
    \item But we also know that  $E\subset \bigcap_i B(S_{X_{i(k)}},\epsilon+\delta_{i(k)})$, which implies (with a little bit of work) that  $E \subset  B(S^*,\epsilon)$.
    \item This in turn implies that $\mathcal{H}^1(S^*) \geq \lambda(\epsilon)$, which, together with Step~\ref{step:boundS*}, implies that $\mathcal{H}^1(S^*) = \lambda(\epsilon)$.
  \end{enumerate}
\end{proof}

\subsection{Case I: Line Segments}
\label{sec:segment}

\begin{lem}\label{lem:linesegment}
  Let $s > 0$, and $[x, y] \subset \R^2$ be the finite line segment of length $L := |x - y|$ with endpoints $x, y \in \R^2$.
  For $E = \B([x, y], s)$
  \begin{align*}
    \sigma(E, s) = \lambda(E, s).
  \end{align*}
\end{lem}
\begin{proof}
  Let $[x, y]\subset \R^2$ be a line segment of length $L$.
  If $E = \B([x, y], s)$, then $[x, y]$ is a minimal length curve for $E$ and $s$.
  Without loss of generality, we may assume that $[x, y]$ is the line segment $[0, L] \subset \{(x, 0) : x \in \R\}$, where we overload the notation $L$ to represent the point $(L,0)$.
	
  For each $n \in \mathbb{N}$, we may dissect $[0, L]$ into $n$ line segments, each of length $L/n$ having endpoints $x_k = kL/n$ for $k = 0, \dots, n$.
  For each $n$, we will construct a closed and connected set $\Gamma_n = [0,L] \, \cup \, P_n$, where $P_n$ consists of what we call {\it prongs}, such that $\Gamma_n$ connects a finite point set $X_n$ with $B(X_n, s) \supset E$.
  This finite point set will consist of $2n + 4$ points, which will be obtained by ``extending'' each $x_k$ ``upward'' and ``downward'', and extending the two end points ``outward'', as shown in Figure~\ref{fig:LineSegmentCase}.
  Since any Steiner tree $S_n := S_{X_n}$ over $X_n$ will, by its very definition, satisfy $\mathcal{H}^1(S_n) \leq \mathcal{H}^1(\Gamma_n)$, if we can show that $\mathcal{H}^1(\Gamma_n) \to L$ as $n\to \infty$, this will imply that we also have $\mathcal{H}^1(S_n) \to L$ as $n\to \infty$.
  Note that we also know that $\mathcal{H}^1(S_n) \geq L$ since $X_n$ will always contain the two endpoints $0$ and $L$.
  This gives us a sequence of Steiner trees which converge to the minimal length.
  In order to show that $\sigma(E, s) = \lambda(E, s)$, we must also know that the $s$-neighborhood of these Steiner trees contain $E$.
	
  Let us construct $\Gamma_n$.
  For $\delta_n > 0$ (to be picked later), and for each $x_k$ $(k = 0, 1, ..., n)$ pick the two points that are $\delta_n$ distance ``above'' and ``below'' $x_i$.
  In other words, let 
  \begin{align*}
    y^i := (x_i, \delta_n) \quad \mathrm{and} \quad y_i := (x_i, -\delta_n).
  \end{align*}
  Now, let $[-\delta_{n}, 0] := \{(x, 0) : -\delta_{n} \leq x \leq 0\}$ and $[L, L + \delta_{n}] := \{(x, 0) : L \leq x \leq L + \delta_{n}\}$ denote the two horizontal line segments, each of length $\delta_{n}$.
  We can now construct,
  \begin{align*}
    \Gamma_n = ([-\delta_{n}, 0]\cup [0, L]\cup [L, L + \delta_{n}]) \,\bigcup \, \left( \mathop{\cup}_{k=0}^n \,[y^k, y_k] \right).
  \end{align*} 
  Note that if $x, y \in \R^2$, $[x, y]$ is simply the closed line segment connecting $x$ and $y$.
  Therefore, $\cup_{k=0}^n \, [y^k, y_k]$ consists of $(n+1)$ vertical line segments of length $2\delta_{n}$. 
	
  Denoting the set of points $\{(-\delta_{n}, 0), (L + \delta_{n}, 0)\}\cup \{y^i, y_i\}_{i = 0}^n$ by $X_n$, we must find $\delta_n > 0$ such that $\B(X_n, s) \supset E$.
  To do this, first notice that
  \begin{align*}
    -\sqrt{s^2 - x^2} \leq \, \frac{x^2}{s} - s \quad \text{on} \quad[-s, s]
  \end{align*}
  and therefore if we let
  \begin{align*}
    \delta_n = \frac{(L/2n)^2}{s},
  \end{align*}
  and pick $n \in \mathbb{N}$ large enough so that $\delta_n < s - \delta_n$, then $\B(X_n, s) \supset E$ (see Figure~\ref{fig:LineSegmentCaseTotal}).
  Now,
  \begin{align*}
    \mathcal{H}^1(\Gamma_n) &= 2(n + 1)\delta_n + 2\delta_n + L\\
    &= 2n\delta_n + 4\delta_n + L\\
    &= 2(n + 2)\frac{L^2}{4s n^2} + L\\
    &\leq C'/n + L
  \end{align*}
  for $C'$ independent of $n$. Therefore $\mathcal{H}^1(\Gamma_n) \to L$ as $n \to \infty$.
\end{proof}

\begin{figure}[htp!]
  \centering \includegraphics[width = 1\linewidth]{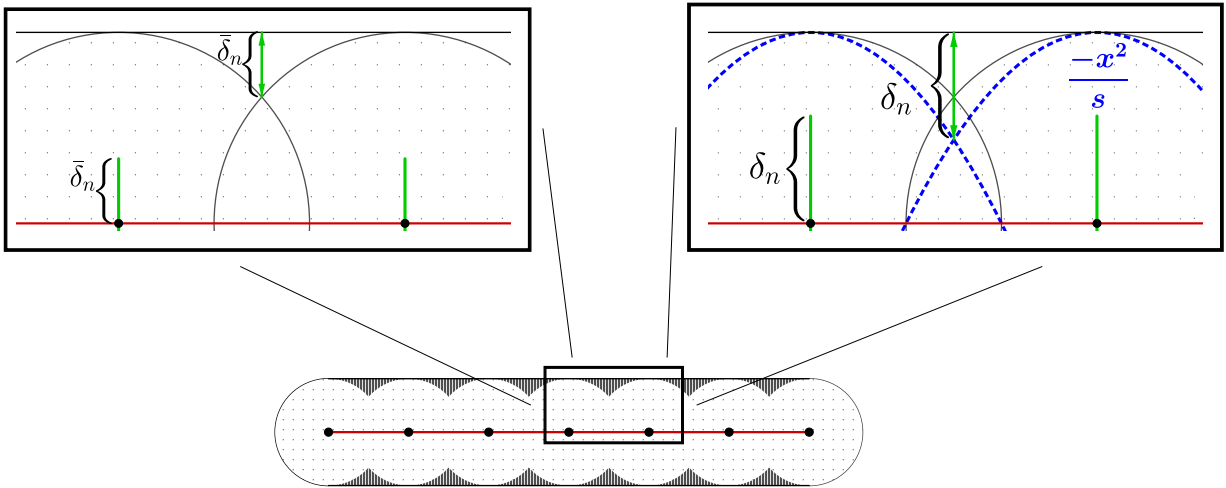}
  \caption{  \label{fig:LineSegmentCaseTotal}
    If we want to cover $E = \B([x, y], s)$ with balls centered on the $2n + 2$ points, we must raise and lower the balls by $\bar{\delta}_{n}$, as is shown on the top-left blown-up picture.
    However, it suffices to extend the balls up by a little more, $\delta_n$, as is shown on the top-right blown-up picture.
    In order to guarantee that raising these balls ``upwards'' and ``downwards'' will not expose the center line, we must choose $\delta_n$ small enough so that $\delta_n < s - \delta_n$.
    Note that we do not extend $(-\delta_n, 0)$ and $(L + \delta_n, 0)$ outwards as in the most general case we consider in this paper.
  }
\end{figure}

\subsection{Case II: $C^1$ Curves}
\label{sec:C1-curves}

\begin{proposition}
  For $E \subset \R^2$ compact and $s > 0$, if $\Gamma \in \Lambda(E, s)$, the set of all $s$-minimizers of $E$, is a $C^1$-curve then
  \begin{align*}
    \sigma(E, s) = \lambda(E, s).
  \end{align*}
\end{proposition}
\begin{proof}
  This proof is an application (with modifications) of the ideas behind Lemma \ref{lem:linesegment}.
  We begin with fact that for any aspect ratio  $\alpha > 0$, there exists a large enough $M \in \N$ such that the partition
  \begin{align*}
    \{t_{i}\}_{i=0}^{M} ~~ \mathrm{ where }~ t_i = \frac{i}{M}  
  \end{align*} 
  of $[0, 1]$ gives us that the images $\gamma([t_i, t_{i + 1}])$ are contained in rectangles $D_i$ (centered along $P^0_i \equiv [\gamma(t_i), \gamma(t_{i + 1})]$, see Figure (\ref{fig:C1_rec_pic})) of width $\mu_i$ and length $\rho_i$ where
  \begin{align*}
    \frac{\mu_i}{\rho_i} < \alpha. 
  \end{align*}
  We will choose $\{\mu_i, \rho_i, \alpha, n_i\}$ later.
  Using our partition, we construct a piecewise linear curve, starting  with
  \begin{align*}
    P^0 = \bigcup_{i = 0}^{M-1} P^0_{i} = \bigcup_{i = 0}^{M-1} [\gamma(t_i), \gamma(t_{i + 1})].
  \end{align*}
  \begin{figure}[htp!]
    \centering 
    \scalebox{.5}{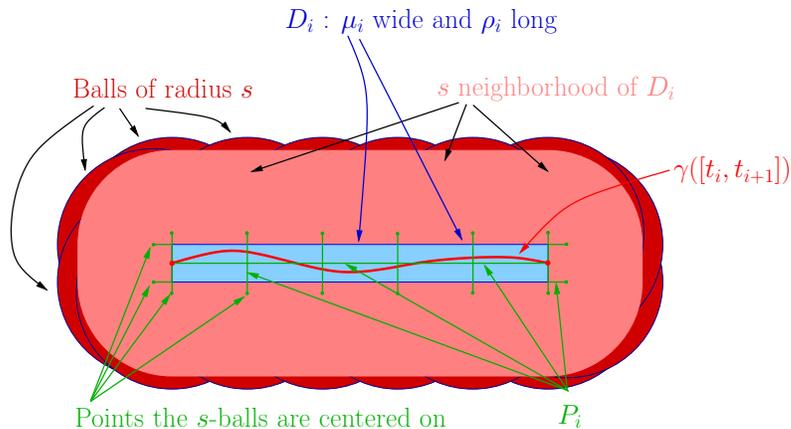}
    \caption{\label{fig:C1_rec_pic}
      $C^1$ rectangle construction.
    }
  \end{figure}
  To each $P^0_i$ we now add prongs $P^1_{i,j}$ pointing up and down, of length $\mu_i/2 \,+ \,\delta_{n_i}$ for $j = 1, \dots, 2n_i+2$.
  We also add $4$ horizontal prongs $P^1_{i,j}$ for $j = 2n_i+3, \dots, 2n_i+6$, two at each end of the rectangle, each of length $\mu_i/2$.
  Centering balls of radius $s$ at each of the free ends of the $2n_i+6$ prongs creates a cover for the $s$-neighborhood of $D_i$.
  We will call the this piecewise linear curve (shown in green in Figure \ref{fig:C1_rec_pic}) $P_i$, and define it precisely as
  \[
     P_i \equiv P^0_i \bigcup \left( \mathop{\cup}_{j=1}^{2n_i+6} \, P^1_{i,j} \right)\,.
  \]
  The complete piecewise linear curve is just $P = \cup_{i=0}^{M-1} P_i$ whose end points (of which there are $\sum_{i=0}^{M-1} 2n_i+6$) are centers of an $s$-neighborhood of $\Gamma$.
  {\bfseries We now show that the excess length of $P$ is as small as you like, provided you choose $\alpha$ small enough.
  }

  \bigskip

  We can assume that $\rho_i < 1$ for all $i$ since choosing $M$ big enough enforces that condition.
  The length of the vertical prongs goes from $\delta_{n_i}$ (analogous to $\delta_n$ in the proof of the previous Lemma \ref{lem:linesegment}) to $\delta_{n_i}+\mu_i/2$.
  And we have added four horizontal prongs of length $\mu_i/2$, so the total length of $P_i$ goes from
  \[ \rho_i + 2(n_i+1)\frac{\rho_i^2}{4s n_i^2} \]
  to
  \begin{equation}  \label{eq:H1Pibnd} 
    \begin{aligned}
      \Hd^1(P_i) &\leq \rho_i + 2(n_i+1)\frac{\rho_i^2}{4s n_i^2} + 2\alpha\rho_i + (n_i+1)\alpha\rho_i\\ 
      &\leq \rho_i\left(1 + 2(n_i+1)\frac{\rho_i}{4s n_i^2} + 2\alpha + 2n_i\alpha\right)\\
      &\leq  \rho_i\left(1 + \frac{\rho_i}{s n_i} + 2\alpha + 2n_i\alpha\right)\\
      &\leq  \rho_i\left(1 + \frac{1}{s n_i} + (2n_i+2)\alpha\right)
    \end{aligned}
  \end{equation}

  At this point, this calculation gives us the length for any choice of $\{\rho_i, \mu_i, \alpha, n_i\}_{i=0}^{M-1}$.
  We will see that choosing $\alpha$ small enough and a universal $n$ (so that $n_i = n$ for all $i$) gets us what we want.
  Here are the steps:
  \begin{enumerate}
  \item Begin by choosing a (small) $\beta>0$.
    We will end up showing that $\Hd^1(P_i) \leq \rho_i(1+2\beta)$.
  \item \label{stp:alpha} Choose $\alpha$ small enough so that $n \equiv \floor*{\frac{\beta}{4\alpha}}$ satisfies $\frac{1}{s n} \leq \beta$.
  \item We note that to make sure the lifted balls cover the central interval, we need to have the lift to be less than $s - \delta_{n_i}$:
    \[ \mu_i/2 + \delta_{n_i} < s-\delta_{n_i}\]
    implying that
    \[ \mu_i/2 + 2\delta_{n_i} < s.\]
    This is analogous to $\delta_n$ shown in Figure~\ref{fig:LineSegmentCaseTotal}.
    Now, because:
    \begin{eqnarray*}
      \mu_i/2 + 2\delta_{n_i} &\leq& \alpha/2 + \frac{\rho_i^2}{2sn_i^2}\\
      &\leq& \alpha/2 + \frac{1}{2sn_i}\\
      &\leq& \alpha + \beta
    \end{eqnarray*}
   it suffices to require that:
    \[\alpha + \beta < s.\]
  \item Define $n_i \equiv n$. Due to the choice of $\alpha$ in Step \ref{stp:alpha},
    \[(2n_i + 2)\alpha \leq 4n_i\alpha \leq \beta.\]  
  \item We also need $\frac{\alpha}{2} < s$.
    In order to have the points at the ends of the 4 horizontal prongs cover the end of the rectangle, we need $\mu_i/2 < s$.
    But since  $\mu_i = \rho_i \alpha$ and we can choose $\rho_i < 1$ for all $i$ this means we want to have $\frac{\alpha}{2} < s$.
  \end{enumerate}
  this allows us to continue Equation (\ref{eq:H1Pibnd}) to get
  \begin{eqnarray*}
    \Hd^1(P_i) &\leq&  \rho_i\left(1 + \frac{1}{s n_i} + (2n_i+2)\alpha\right)\\
    &\leq& \rho_i\left(1 + \beta + \beta \right) = \rho_i(1+2\beta).
  \end{eqnarray*}
Because we know that  $\sum_{i=0}^{M-1} \rho_i \leq \Hd^1(\Gamma)$, we conclude
\begin{eqnarray*}
  \Hd^1(P) &\leq& \sum_{i=0}^{M-1} \Hd^1(P_i)\\
           &\leq& (1+2\beta) \sum_{i=0}^{M-1} \rho_i\\
           &\leq&  (1+2\beta) \Hd^1(\Gamma).
\end{eqnarray*}
\end{proof}


\subsection{Case III: Lipschitz curves}
\label{sec:lip-case}

\begin{lem}\label{lem:main}
  Let $s > 0$ and $\Gamma\subset \R^2$ be a Lipschitz curve of positive length.
  Then given $\epsilon > 0$, there exists a finite point set $X := \{x_i\}_{i = 1}^N\subset \R^2$ and a Lipschitz curve $\Gamma_\ast$ that contains $X$ such that 
  \begin{align*}
    B(X, s) \supset B(\Gamma, s) \quad \mathit{and}\quad \mathcal{H}^1(\Gamma) \leq \mathcal{H}^1(\Gamma_\ast) \leq \mathcal{H}^1(\Gamma) + \epsilon. 
  \end{align*}
\end{lem}
\begin{proof}
  Consider an arc-length parameterization $\gamma:[0, L] \to \R^2$ of $\Gamma$ where $L = \mathrm{length}(\Gamma)$. 
  
  Using this parameterization, we will construct such a closed and connected $\Gamma_\ast$ by adding extra small line segments to particular places of $\Gamma$.
  Precisely how we add these extra line segments will depend on whether we are centered around a {\it good} portion of $\Gamma$ or a {\it bad} portion of $\Gamma$.
  Because $\gamma$ is Lipschitz, most of $\Gamma$ will be a good portion, and we must therefore have tight control on how exactly we are adding these extra line segments.
  The line segments around these good portions will be denoted as $P$, and will be called {\it prongs}.
  In contrast, the bad portions of $\Gamma$ will be small, and will allow us to be more liberal in how we add the extra line segments around them.
  The line segments around these bad portions will be denoted as $S$, and will be called {\it spokes}.
  We will then define
  \begin{align*}
    \Gamma_\ast := \Gamma \cup P \cup S.
  \end{align*}
  Since $\gamma$ is Lipschitz, the set of differentiable points of $\gamma$, $G\subset I := [0, L]$ has full measure in $I$.
  We will call any $x \in G$ a {\it{good}} point, and any $x \in I\setminus G$ a {\it{bad}} point.
  The image around any good point will be contained in a cone whose aspect ratio will go to $0$.
  Precisely, for any $x \in G$
  \begin{align*}
    \frac{\eta(x, r, \rho(x, r))}{\rho(x, r)} \to 0\ \mathrm{as}\ r\to 0
  \end{align*}
  where
  \begin{align*}
    \rho(x, r) &= \inf\{t : S(y, \Tan(\Gamma, y), t)\supset \gamma(x - r, x + r)\}\ \mathrm{and}\ \\
    \eta(x, r, s) &= \inf\{h : C(y, \Tan(\Gamma, y), h/s; s) \supset \gamma(x - r, x + r)\}.
  \end{align*}
  If $x$ and $r$ is clear from context, we will simply refer to the above aspect ratio as $\eta/\rho$.
  Recall that Definition \ref{del:orthcplmt} introduces the sets $S$ and $C$.

  Given a small aspect ratio $\alpha > 0$, we will construct a particular partition of $I$ as follows.
  For any $\xi > 0$, since $G$ is Borel we may pick a compact subset $K$ of $G$ such that $\mathcal{L}^1(G\setminus K) < \xi$.
  So that the constants work out at the end, we let 
  \begin{align*}
    \xi := \frac{\epsilon}{16\mathrm{Lip}(\gamma)}.
  \end{align*}
  Differentiability of $\gamma$ in $G$ implies~\cite[\S3.1.21]{federer-1969-1} that for any $x \in K$ there is a small enough $R > 0$ such that for any $r \leq R$
  \begin{align*}
    \frac{\eta}{\rho} < \alpha.
  \end{align*} 
  Without loss of generality, we may assume that $\rho < 1$ and $\eta < s/100$.
  Therefore from the open cover $\mathcal{G}_\alpha = \{\U(x, R)\}_{x \in K}$ of $K$, we may extract a {\it good} finite subcover
  \begin{align*}
    \mathcal{G} = \{\U(x_i, R_i)\}_{i = 1}^{N}
  \end{align*} 
  of $K$ for which we may assume that no $\U(x_i, R_i)$ is contained in any other $\U(x_j, R_j)$.
  Since $I\setminus K$ is equal to a finite union of disjoint, connected, closed subintervals $\{\B(b_i, \xi_i)\}_{i = 1}^M$ for some $b_i \in I\setminus K$ and $\xi_i \geq 0$,
  we can define the corresponding {\it bad} finite cover
  \begin{align*}
    \mathcal{B} = \left\{\B(b_i, \epsilon_i)\right\}_{i = 1}^M
  \end{align*}
  of $I\setminus K$.
  We call $\{x_i\}_{i = 1}^N$ and $\{b_i\}_{i = 1}^M$ the set of good centers and the set of bad centers, respectively.
  Note that the way we chose our bad cover implies that there is always at least one good center in between any two bad centers.
  Also note that $\sum_{i = 1}^M 2\xi_i \leq \xi$.

  Let us now order all the good and bad centers $Z = \{z_i\}_{i=1}^{N+M} = \{x_i\}_{i=1}^N\cup \{b_i\}_{i=1}^M$ by their natural ordering in $\R$.
  In what follows, we will obtain a $u_i$ in between each $z_i$ and $z_{i + 1}$.
  Let $V_i$ and $V_{i+1}$ in $\mathcal{G} \cup \mathcal{B}$ be the cover elements corresponding to $z_i$ and $z_{i+1}$, respectively.
  If both $z_i$ and $z_{i+1}$ are good centers then $V_i \cap V_{i+1} \neq \emptyset$ and we may therefore pick a point $u_i \in V_i \cap V_{i+1}$ such that $z_i < u_i < z_{i+1}$.
  If $z_i$ is a good center and $z_{i+1}$ is a bad center let $u_i = {\mathbf{cl}}(V_i) \cap V_{i+1}$.
  Similarly, if $z_i$ is a bad center and $z_{i+1}$ is a good center, we let $u_i = V_i \cap {\mathbf{cl}}(V_{i+1})$.
  Lastly, to deal with the endpoints we will let $u_0 = 0$ and $u_{N + M} = 1$.
  For what follows it is important to note that
  \begin{align*}
    [u_i, u_{i+1}] &\subset {\bf{cl}}(V_{i+1})\ \mathrm{for\ any}\ i = 1, ..., N + M - 2,
  \end{align*}
  and that the map 
  \begin{align*}
    Z   & \to \{[u_{i-1}, u_i]\}_{i=1}^{N+M}\\
    z_i & \mapsto [u_{i-1}, u_i]
  \end{align*}
  is a bijection.
  This allows us to partition $\Gamma$ into good parts and bad parts, with the images of the intervals corresponding to good points and bad points, respectively.

  {\bf Covering good parts with prongs:} For $j = 1, \dots, N$ let $z_i := z_{i(j)}$ be a good center, let $Z_i := [u_{i - 1}, u_i]$ be its corresponding interval, and let $y_i := \gamma(z_i)$.
  Instead of considering the symmetric cone $C(y_i, \mathrm{Tan}(\Gamma, y_i), \eta_i/\rho_i; \rho_i)$ that contains $\gamma(\B(z_i, R_i))$, we will instead shorten this cone horizontally, as much as possible, while still containing $\gamma(Z_i)$. 

  Assuming that $y_i = 0$ and that $\mathrm{Tan}(\Gamma, y_i) = \{(x, 0) : x \in \R\}$, we define this cone as
  \begin{align}\label{def:shortcone}
    C_i := C(y_i, \mathrm{Tan}(\Gamma, y_i), \eta_i/\rho_i; [\nu_i, \tau_i])
  \end{align}
  where
  \begin{align*}
    \nu_i = \inf\mathrm{Tan}(\Gamma, y_i)_\sharp[\gamma(Z_i)]\quad \mathrm{and}\quad \tau_i = \sup\mathrm{Tan}(\Gamma, y_i)_\sharp[\gamma(Z_i)]. 
  \end{align*}
  \begin{figure}[htp!]
    \centering \includegraphics[width=4in]{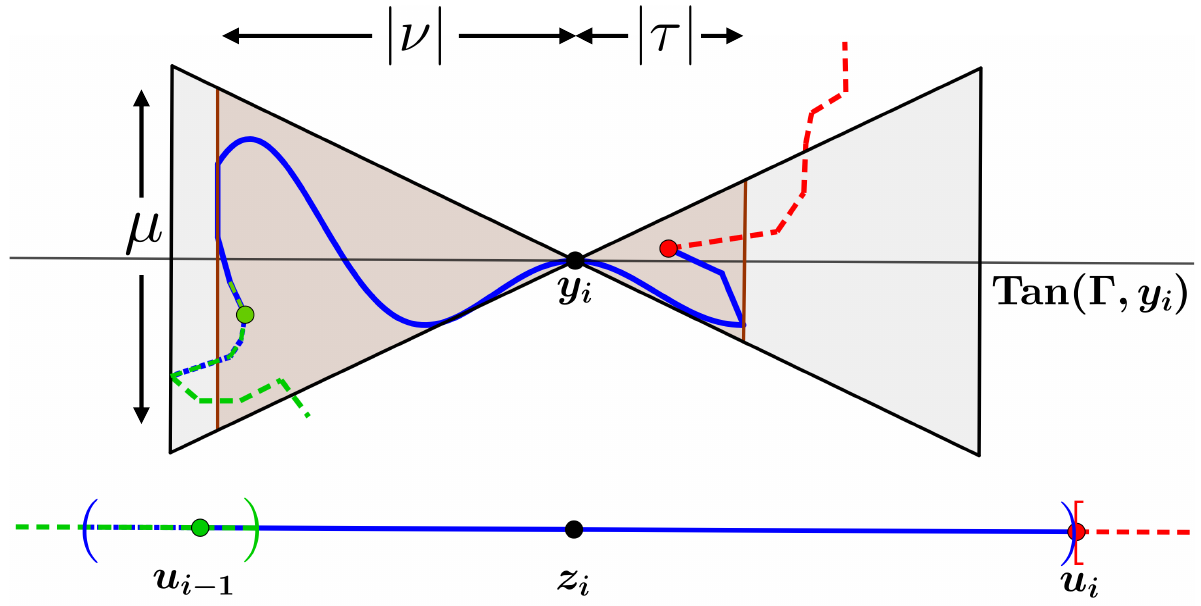}
    \caption{\label{fig:ImageInsideCone}
      The cone is shortened asymmetrically so that its ends intersect $\gamma(Z_i)$.
      Notice that $\mu/(|\nu| + |\tau|)$ is at most $2\alpha$.}
  \end{figure}

  Denote the height of the tallest side of this new cone as $\mu_i := 2\eta_i \max\{|\nu_i|, |\tau_i|\}/\rho_i$, and the width as $\bar{\rho_i} := |\nu_i - \tau_i|$ (see Figure~\ref{fig:ImageInsideCone}).
  Since $\gamma(Z_i)$ is closed, and is contained in the cone $C_i$, we have that $\Tan(\Gamma, y_i)_\sharp[\gamma(Z_i)] = \{(x, 0) : \nu_i \leq x \leq \tau_i\}$.
  Therefore, any line segment at least as long as $2\mu_i$ that is centered on and is perpendicular to $\Tan(\Gamma, y_i)_\sharp[\gamma(Z_i)]$ will also intersect $\gamma(Z_i)$.  

  Still focusing our attention around a good piece $\gamma(Z_i)$, we will construct a finite point set $X^n_i$ where 
  \begin{align}\label{eq:containment}
    B(X_i^n, s)\supset B(\gamma(Z_i), s).
  \end{align}
  Since the cone $C_i$ contains $\gamma(Z_i)$, we will guarantee the above result by showing that $B(X_i^n, s)\supset B(R_i, s) \supset B(C_i, s)$, where $R_i$ is the smallest rectangle that contains $C_i$ as is depicted in Figure (\ref{fig:ConeAndRectangle}), i.e.,
  \begin{align*}
    R_i := \{(x, y) : \nu_i \leq x \leq \tau_i, -\mu_i \leq 2y \leq \mu_i\}.
  \end{align*}
  The points in $X_i^n$ will then be connected by $\gamma(Z_i)\cup P_i$, where $P_i$ will consist of $n + 1$ number of equally spaced line segments of small enough length, all perpendicular to the tangent, together with four other short line segments.

  For each $i = 1,\dots, N + M$ let $n_i \in \N$ and define $\delta_{n_i} := (\bar{\rho}_i/2n_i)^2/s$. 
  We define
  \begin{align*}
    X_i^n := \left(\left\{\nu_i - \delta_{n_i}, \tau_i + \delta_{n_i}\right\}\times \left\{\pm\frac{\mu_i}{2}\right\}\right) \bigcup \left(\left\{\nu_i + k\frac{\bar{\rho_i}}{n_i}\right\}_{k = 0}^{n_i} \times \left\{\pm(\delta_{n_i} + \frac{\mu_i}{2})\right\}\right).
  \end{align*}
  Given the points in $X_i^{n_i}$, we then define the prongs 
  \begin{align*}
    P_i := \underbrace{([\tau_i, \tau_i + \delta_{n_i}]\cup[\nu_i - \delta_{n_i}, \nu_i])\times \{\pm\mu_i/2\}}_{\text{horizontal line segments}} \, \bigcup \, \underbrace{\{\nu_i + k\bar{\rho_i}/n_i\}_{k = 0}^{n_i}\times [-(\delta_{n_i} + \mu_i/2), \delta_{n_i} + \mu_i/2]}_{\text{vertical line segments}}.
  \end{align*}
  Note that
  \begin{align*}
    \mathcal{H}^1(P_i) &= 4\delta_{n_i} + 2 (n_i + 1)(\mu_i/2 + \delta_{n_i}).
  \end{align*}
  \begin{figure}[htp!]
    \centering
    \includegraphics[width = 6.2in]{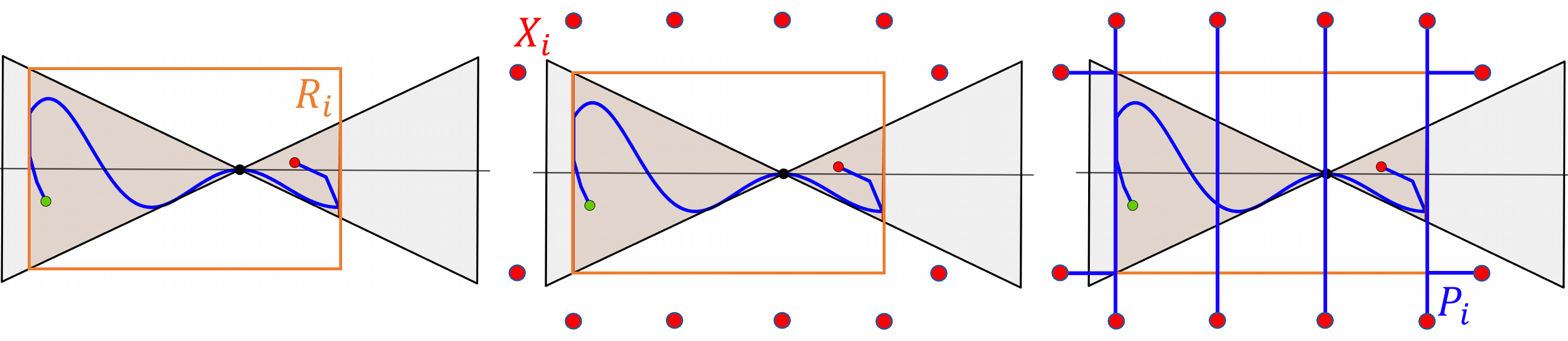}
    \caption{The rectangle of height $\mu_i$ and width $\bar{\rho}_i$, the finite point set $X_i$ such that $B(R_i, s) \subset B(X_i, s)$, and the prongs $P_i$ connecting $X_i$ to $\gamma(Z_i)$.}
    \label{fig:ConeAndRectangle}
  \end{figure}

  To show that $\gamma(Z_i)\cup P_i$ is connected, notice that each vertical line segment in $\{\nu_i + k\bar{\rho_i}/n_i\}_{k = 0}^{n_i}\times [-(\delta_{n_i, i} + \mu_i/2), \delta_{n_i} + \mu_i/2]$ of length $2\delta_{n_i} + \mu_i$ is centered on, and is perpendicular to $\Tan(\Gamma, y_i)_\sharp[\gamma(Z_i)]$.
  Also, each of the 4 line segments in $([\tau_i, \tau_i + \delta_{n_i}]\cup [\nu_i - \delta_{n_i}, \nu_i])\times \{\pm\mu_i/2\}$ is connected to some vertical line segment.

  In order for $B(X_i^{n_i}, s) \supset B(R_i, s)$, we require that $\mu_i + \delta_{n_i} < s - \delta_{n_i}$.
  This is guaranteed when $n_i$ is chosen large enough so that $1/s < n_i$.
  However, we will choose $n_i$ with more precision later.     

  {\bf Covering bad part with spokes:}
  For $j = 1, \dots, M$, let $z_i := z_{i(j)}$ be a bad point, let $Z_i := [u_{i - 1}, u_i]$ be its corresponding interval, and let $y_i := \gamma(z_i)$.
  We now construct the spokes $S_i$ connecting sets of points $Y_i$ such that 
  \begin{align}\label{eq:badcover}
    B(Y_i, s) \supset B(\gamma(Z_i), s).
  \end{align}
  Each of these spokes, $S_i$ will consist of line-segments emanating from the image of the corresponding bad center.
  The length of these line-segments will be bounded above by the length of the bad center's interval and $\text{Lip}(\gamma)$.

  Recalling that bad intervals $Z_i = B(z_i, \xi_i)$, we get that $\gamma[B(z_i, \xi_i)] \subset B(\gamma(z_i), \mathrm{Lip}(\gamma)\xi_i)$ implies $\gamma(Z_i) \subset B(\gamma(z_i), \mathrm{Lip}(\gamma)\xi_i)$.
  Therefore, constructing a $Y_i$ such that 
  \begin{align*}
    B(Y_i, s) \supset B(\gamma(z_i), \mathrm{Lip}(\gamma)\xi_i + s)
  \end{align*}
  will give us the result in Equation (\ref{eq:badcover}).
  To this end, we simply define $Y_i$ to be
  \begin{align*}
    Y_i := \gamma(z_i) + \{(0, \pm 2\mathrm{Lip}(\gamma)\xi_i), (\pm 2\mathrm{Lip}(\gamma)\xi_i, 0)\}.
  \end{align*}
  Now, the collection of spokes $S_i$ consists simply of line segments connecting every point in $Y_i$ to the center $\gamma(z_i)$.
  Precisely, 
  \begin{align*}
    S_i := \gamma(z_i) + (([-2\mathrm{Lip}(\gamma)\xi_i, 2\mathrm{Lip}(\gamma)\xi_i]\times \{0\}) \cup (\{0\}\times[-2\mathrm{Lip}(\gamma)\xi_i, 2\mathrm{Lip}(\gamma)\xi_i])).
  \end{align*}
  Note that $\mathcal{H}^1(S_i) = 8\mathrm{Lip}(\gamma)\xi_i$.
  See Figure (\ref{fig:CoveringBadImageWithSpokes}) for an illustration of this step.
  It is clear that $\Gamma \cup S_i$ is connected since $\gamma(z_i)$ is in $S_i$ and $\Gamma$.
  \begin{figure}[h]
    \centering
    \includegraphics[width=2.5in]{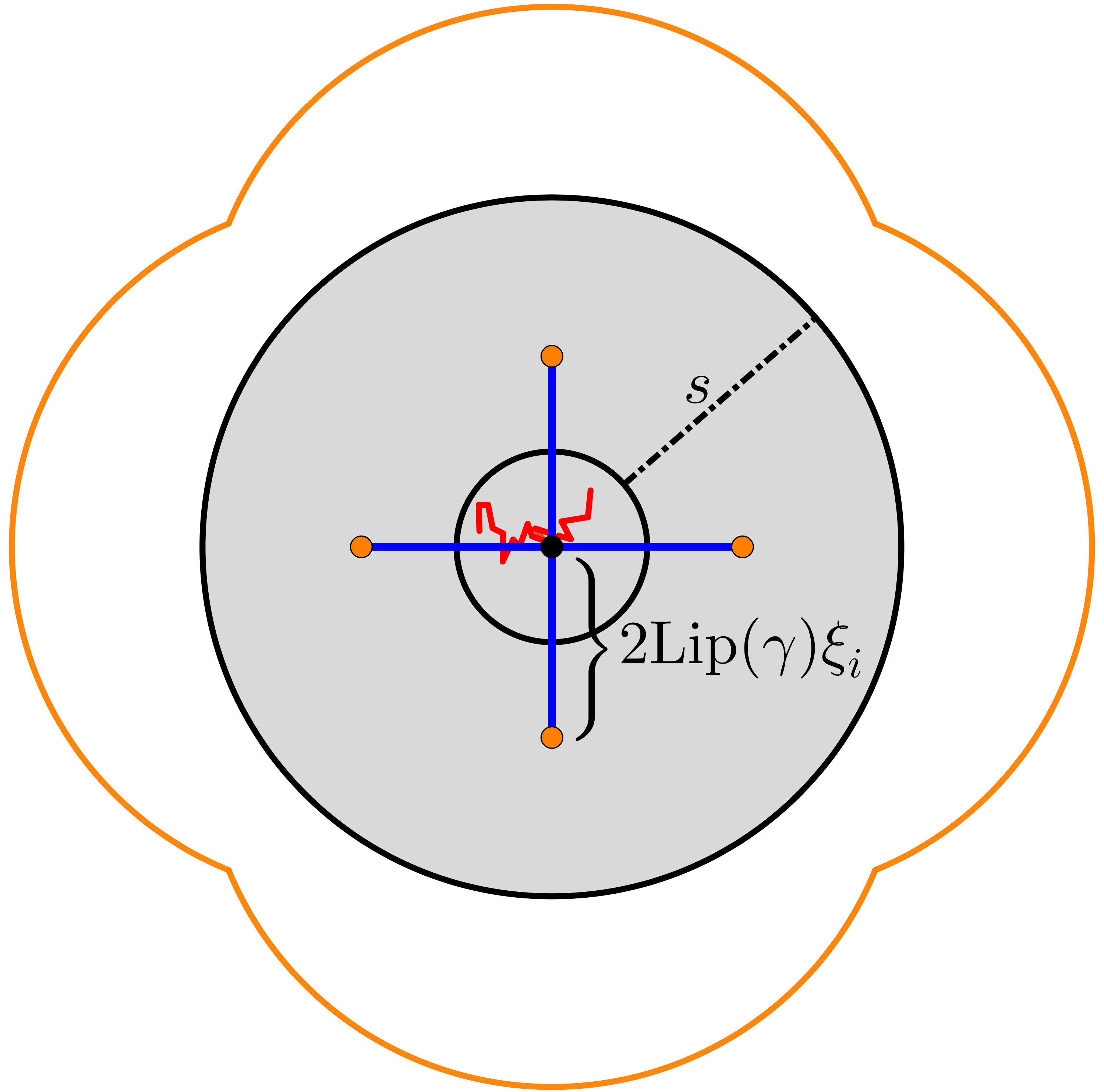}
    \caption{  \label{fig:CoveringBadImageWithSpokes}
      Illustration of the cover of a part of the bad subset of $\Gamma$ with spokes.
      The boundary of the four balls that cover the $s$-neighborhood of the small ball covering the bad piece is shown in yellow.
    }
  \end{figure}

  {\bf Estimating $\mathcal{H}^1(\Gamma\cup P\cup S)$:}
  We now find suitable upper bounds for $\mathcal{H}^1(\Gamma \cup P\cup S) \leq \mathcal{H}^1(\Gamma) + \mathcal{H}^1(P) + \mathcal{H}^1(S)$, where
  \begin{align}
    \mathcal{H}^1(P) = \sum_{i = 1}^N 4\delta_{n_i} + 2(n_i + 1)(\mu_i/2 + \delta_{n_i})\quad \mathrm{and} \quad \mathcal{H}^1(S) = \sum_{i = 1}^M 8\mathrm{Lip}(\gamma)\xi_i.
  \end{align}
  First, by our initial choice of $Y_i$, we simply have that 
  \begin{align*}
    \mathcal{H}^{1}(S) &\leq 8\mathrm{Lip}(\gamma)\xi \leq \frac{\epsilon}{2}.
  \end{align*}
  As for the first term, we let $0 < \beta \leq \epsilon/(4L)$.
  We will show the existence of $\alpha > 0$ and $n_i := n > 0$ so that 
  \begin{align*}
    \mathcal{H}^1(P) \leq \beta L \leq \frac{\epsilon}{2}.
  \end{align*}
  First, recall that given $\alpha > 0$, since for each $i = 1, \dots, N$, we have that $\mu_i/\rho_i < \alpha$ (hence $\mu_i/\bar{\rho}_i < 2\alpha$), and that for any $n \in \mathbb{N}$, 
  \begin{align*}
    \mathcal{H}^1(P_i) = 4\delta_{n_i} + 2(n_i + 1)(\mu_i/2 + \delta_{n_i}).
  \end{align*}
  We first make sure that $\alpha > 0$ is picked small enough so that $n := \floor{\beta/12\alpha}$ satisfies the two conditions
  \begin{align*}
    \frac{1}{sn} < \beta\quad \mathrm{and} \quad\frac{1}{s} < n.
  \end{align*} 

  For all $i = 1, \dots, N$, we let $n_i := n$.
  The first condition will give us Inequality~(\ref{eq:firstcond}), and the second condition implies that $\delta_{n_i} \leq 2\alpha\bar{\rho}_i$ and also that $\mu_i + \delta_{n_i} < s - \delta_{n_i}$ (and hence $B(X_i^{n_i}, s) \supset B(R_i, s)$).
  Therefore,   
  \begin{align}
    \mathcal{H}^1(P_i) &= 4\delta_{n_i} + 2(n_i + 1)(\mu_i/2 + \delta_{n_i})\\
    &\leq 4(2\alpha\bar{\rho}_i) + 2(n_i + 1)\left(\alpha\bar{\rho}_i + \frac{\bar{\rho}^2_i}{4sn_i^2}\right)\\
    &= \bar{\rho}_i\left(8\alpha + 2(n_i + 1)\alpha + \frac{(n_i + 1)\bar{\rho}_i}{2s n_i^2}\right)\\
    &\leq \bar{\rho}_i\left(8\alpha + 4n_i\alpha + \frac{\bar{\rho}_i}{s n_i}\right)\\
    &\leq \bar{\rho}_i\left(8\alpha + 4n_i\alpha + \bar{\rho}_i\beta\right)\label{eq:firstcond}\\
    &\leq \bar{\rho}_i(12\alpha n_i + \bar{\rho}_i\beta)\\
    &\leq \bar{\rho}_i(\beta + \bar{\rho}_i\beta)\\
    &\leq 2\bar{\rho}_i\beta.
  \end{align}
  Therefore, we can now see that 
  \begin{align*}
    \mathcal{H}^1(P) &\leq \sum_{i = 1}^N \mathcal{H}^1(P_i)\\
    &= 2\beta\sum_{i = 1}^N \bar{\rho}_i\\
    &\leq 2\beta L\\
    &\leq \frac{\epsilon}{2}.
  \end{align*}
  Putting everything together, we get that 
  \begin{align*}
    \mathcal{H}^1(\Gamma_\ast) &\leq \mathcal{H}^1(\Gamma) + 2\beta L + 8\mathrm{Lip}(\gamma)\xi\\
    &= \mathcal{H}^1(\Gamma) + \epsilon.
  \end{align*}
  
  \vspace*{-0.3in}
\end{proof}

\begin{remark}\label{rem:MPS}
  The key difference between the techniques used by Miranda Jr.~et al.~\cite{miranda2006one} and ones we use in Lemma~\ref{lem:main} is that we use a parameterization $\gamma : I \to \mathbb{R}^2$ of $\Gamma$, whereas they work only with its {\it image} $\Gamma$. \\ \vspace*{-0.05in} \\ 
In addition, we provide explicit locations for a finite number of points whose $s$-balls cover $B(\Gamma, s)$.
In contrast, Miranda Jr.~et al~provide a new curve $\Gamma_\ast$ with a potentially smaller neighborhood that will cover $B(\Gamma, s)$. \\ \vspace*{-0.05in} \\ 
Rather than partitioning $\Gamma$ with the images of good and bad portions of the domain of $\gamma$ obtained with Rademacher's theorem, Miranda Jr.~et al~\cite{miranda2006one} partition the image of $\gamma$ by applying Egorov's theorem to a sequence of functions $\beta_k : \Gamma \to \mathbb{R}$ that is meant to capture \emph{the flatness of $\Gamma$ around $x \in \Gamma$ at scale $k^{-1}$}. 
These functions are defined as
\[ \beta_k(x) = \inf_{\Pi} \sup_{y \in \Gamma \cap B(x, k^{-1})} \frac{\mathrm{dist}(y, \Pi)}{k^{-1}}\]
where $\Pi$ is a line containing $x$. \\  \\ 
Our use of a parameterization provides a certain benefit in the case where $\gamma$ is injective.  
In particular, their choice to use two vertical line segments per rectangle as opposed to our choice of many, requires them to use more rectangles, and in turn, approximately double the number of necessary line segments.
We illustrate this difference in the case of $\Gamma$ being a line segment as in Lemma~\ref{lem:linesegment}. 
Partitioning the line segment for $N = 3$, our method would provide an excess length of $8\delta$, whereas using their method with $N=3$ rectangles would require $12\delta$ 
(see Figure~\ref{fig:comparison}).
In general, excess length using their method will be higher by $(N-1)2\delta$. \\ \vspace*{-0.05in} \\ 
Stepanov and Paolini ~\cite{paolini2004qualitative} proved a generalization of Theorem 3.7 of Miranda Jr.~et al.~\cite{miranda2006one} to the case of continua in $\mathbb{R}^n$ using a similar construction to one used by the latter. 
\end{remark}

\begin{figure}[htp!]
    \centering
    \includegraphics[width = 3in]{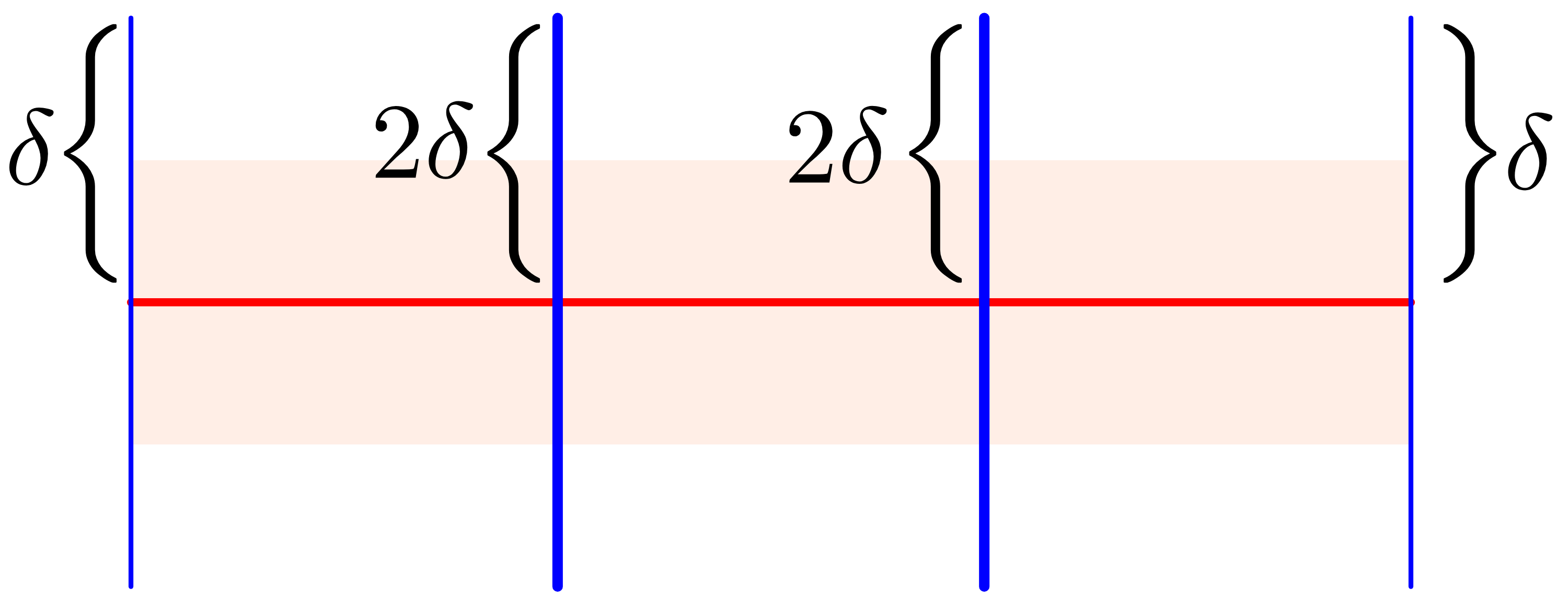}
    \caption{\label{fig:comparison}
        Comparing the method we use in proof of Lemma~\ref{lem:linesegment} with that of Miranda Jr.~et al.~\cite{miranda2006one} for the case when $\Gamma$ is a line segment.
        Note that the middle two segments are shown thicker to indicate the two copies of adjacent line segments as used in the latter approach.}
\end{figure}

\subsection{Case IV: Finite Continua}

In the most general case when $\Gamma \subset \R^2$ is a finite continuum, Lemma~\ref{lem:main} still holds.
As mentioned in Remark~\ref{rem:equivalence}, this is due to the important fact that finite continua, 1-rectifiable continua, and Lipschitz curves are all equivalent. 

\begin{thm}\label{thm:main}
  Let $E\subset \R^2$ be compact and let $s > 0$. Then 
  \begin{align*}
    \sigma(E, s) = \lambda(E, s).
  \end{align*}
\end{thm}

\begin{proof}
  By the existence result in Theorem~\ref{thm:existence}, we know that a minimizer $\Gamma$ of $\lambda(E, s)$ is a compact, connected set such that $\mathcal{H}^1(\Gamma) < +\infty$.
  Letting $\epsilon > 0$, by Lemma~\ref{lem:main} there exists a finite point set $X_\epsilon\subset \R^2$ and a compact and connected $\Gamma_\epsilon$ containing $X_\epsilon$ such that
  \begin{align*}
    B(\Gamma, s) \subset B(X_\epsilon, s)\quad \mathrm{and}\quad \mathcal{H}^1(\Gamma_\epsilon) \leq \mathcal{H}^1(\Gamma) + \epsilon.
  \end{align*}
  In particular, any Steiner tree $S_{X_\epsilon}$ over $X_\epsilon$ will be a candidate minimizer for $\sigma(E, s)$ and $\lambda(E, s)$ and will satisfy
  \begin{align*}
    \mathcal{H}^1(\Gamma) \leq \mathcal{H}^1(S_{X_\epsilon}) \leq \mathcal{H}^1(\Gamma_\epsilon).
  \end{align*}
  Therefore we get that $\mathcal{H}^1(S_{X_\epsilon}) \leq \mathcal{H}^1(\Gamma) + \epsilon$.
  Letting $\epsilon \to 0$ proves our theorem.
\end{proof}

The final corollary follows from Remark~\ref{rem:main}.
It says that when we define $\sigma(E, s)$, instead of taking Steiner trees over $X$, we can take minimum spanning trees over $X$, and get the same result of Theorem~\ref{thm:main}.

\begin{corollary}\label{cor:main}
  Let $E\subset \R^2$ be compact and let $s > 0$.
  Define the analogous 
  \begin{align*}
    \sigma'(E, s) :=& \inf\{\mathcal{H}^1(T_X) : X = \{x_i\}_{i = 1}^N,\ B(X, s) \supset E\}
  \end{align*}
  where we take minimum spanning trees $T_X$ over $X$, instead of Steiner trees.
  Then 
  \begin{align*}
    \sigma'(E, s) = \lambda(E, s).
  \end{align*}
\end{corollary}
\begin{proof}
  Given any Steiner tree $S_X$ over a finite point set $X$, there exist a finite number of Steiner points $X'$.
  Then for any minimum spanning tree $T_{X\cup X'}$ over $X\cup X'$, we get that $\mathcal{H}^1(T_{X\cup X'}) = \mathcal{H}^1(S_X)$.
  Therefore
  \begin{align*}
    \sigma'(E, s) = \sigma(E, s),
  \end{align*}
  and we get $\sigma'(E, s) = \lambda(E, s)$ by Theorem~\ref{thm:main}. 
\end{proof}

\section{Computational Exploration} \label{sec:compexplr}
\vspace*{-0.025in}

We have implemented in Python a framework for computational exploration of the maximum distance problem in $\R^2$ using minimum spanning trees.
The framework is available as open source at \href{https://github.com/mtdaydream/MDP_MST}{https://github.com/mtdaydream/MDP\_MST}.
We employ functions from the Shapely package \cite{shapely} for most of the geometric operations.
Minimum spanning trees are computed using our implementation of Kruskal's algorithm \cite{Kr1956}.
Sample output from the package is shown in Figure \ref{fig:compute}.

The domain $E$ is specified by a sequence of points on its boundary, defined by the union of edges connecting consecutive pairs of the points.
Disjoint holes are allowed in $E$.
For a nominal radius $s$, the user then chooses vertices for the minimum spanning tree $T$.
Coverage of $E$ by the minimum spanning tree $T+s$-ball is verified and displayed.
The user can then vary the value of $s$ while keeping $T$ fixed.
Alternatively, they could choose a different set of vertices for a new MST.

\begin{figure}[hbp!]
  \vspace*{0.02in}
  \hspace*{0.05in}
  \includegraphics[width=3in]{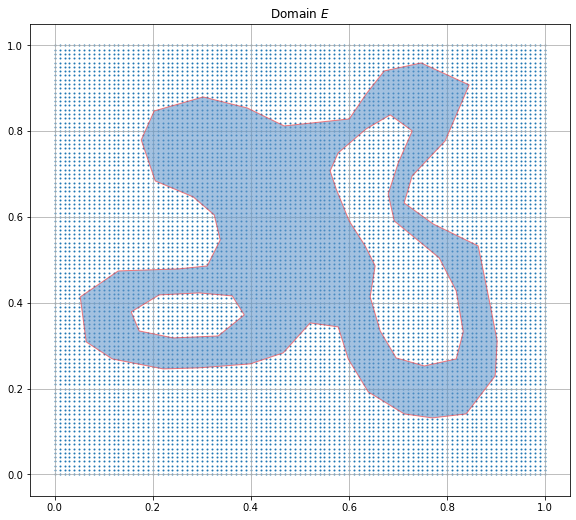}
  \hspace*{0.22in}
  \includegraphics[width=2.95in]{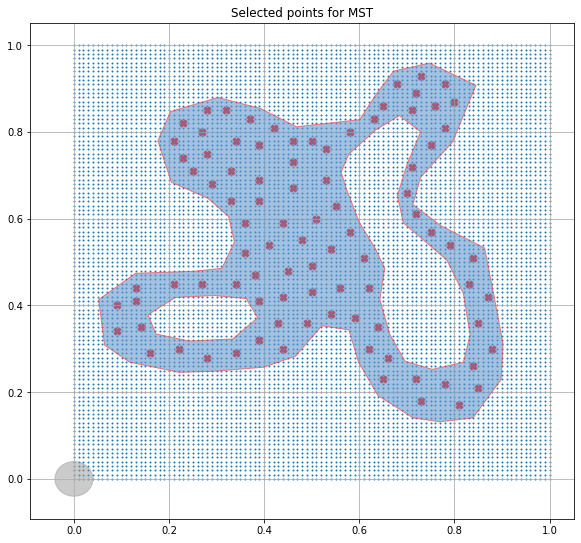}\\
  \includegraphics[width=3.08in]{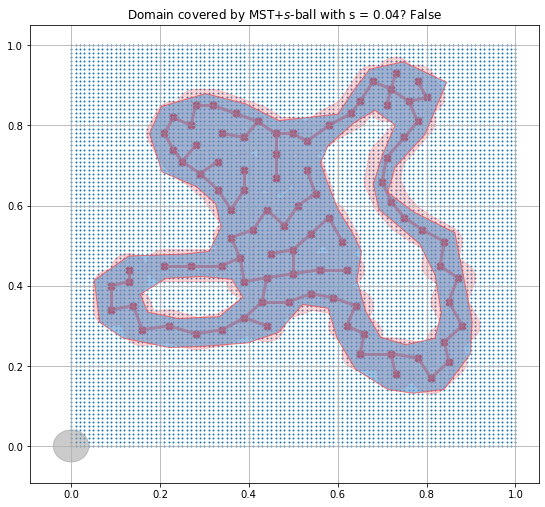}
  \hspace*{0.2in}
  \includegraphics[width=3in]{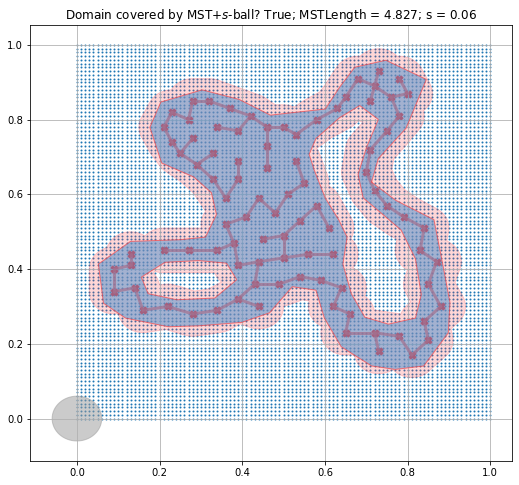}
  \caption{\label{fig:compute}
    A closed domain $E$ with two disjoint holes (top left), and the set of vertices selected for the minimum spanning tree (MST) with a nominal radius of $s=0.04$ (top right).
    But the  MST $T+s$-ball does not cover the domain $E$ (bottom left).
    Instead, with $s=0.06$ for instance, we do get the same $T+s$-ball covering $E$ (bottom right).
    Length of the MST is $\Hd^1(T)=4.827$.
  }
\end{figure}

\section{Discussion} \label{sec:disc}
We have shown that for compact sets $E \subset \R^2$, solving the  maximum distance problem for $s > 0$ by minimizing over continua whose $s$-neighborhoods cover $E$, reduces to simply minimizing over finite collections of balls of radius $s$ which cover $E$.
In the proof of our main theorem, we use knowledge of a minimizer to construct the covering of $E$ with balls. 

Motivated in part by the related traveling salesman problem with neighborhoods \cite{deBGuKaLeOvvdS2005} and approximation schemes for the same \cite{AnFlHoSc2019}, we could investigate ways to approximate $\sigma(E, s)$ without knowledge of minimizers. 
For instance, we could consider finite dimensional approximations of $\sigma(E, s)$ where we are allowed to use only $n$ number of balls of radius $s$ to cover $E$, for some fixed $n \in \mathbb{N}$.
Precisely, for each $n \in \mathbb{N}$, we look at the topological subspaces $\mathcal{M}_n$ of the \emph{$n$-th unordered configuration spaces of $\R^n$}, i.e., we define $\mathcal{M}_n$ to be
\[\{(x_1, \dots, x_n) \in \R^{2n} : \cup_{i = 1}^nB(x_i, s) \supset E\} \,\setminus\, \{(x_1, \dots, x_n) : x_i \neq x_j\ \text{for some}\ i\neq j\}\]
{\it modulo} the action of the symmetry group of order $n$ on the indices of the coordinates $(x_1, \dots, x_n)$. 
Note that $x_i$ represents the 2D coordinates of the $i$-th point.
We observe that $\mathcal{M}_n$ is also a topological subspace of the \emph{a priori} infinite dimensional space $\cup_n \mathcal{M}_n$ we used to compute $\sigma(E, s)$. 
We may then investigate the analogous minimization problem 
\[\sigma_n(E, s) := \{\mathcal{H}^1(S_X) : X \in \mathcal{M}_n\}\]
as a finite dimensional approximation of $\sigma$. 
Our results show that computing $\sigma_n$ is a reasonable approximation of $\sigma$, and hence also of $\lambda$, since 
\[\lim_{n\to \infty}\sigma_n = \sigma = \lambda.\]
Following Corollary \ref{cor:main}, we may consider the minimal spanning tree (MST) $T_X$ in place of the Steiner tree $S_X$, as it is more efficient to compute MSTs.
We may take a finite subspace $\mathcal{X} \subset \mathcal{M}_n$ specified by a finite sample, and compute the minimal spanning tree $T_X$ over each sample $X \in \mathcal{X}$.

As a alternative approach, we want to study sublevel set persistent homology \cite{CoEdHa2007,EdHa2009} of the Vietoris-Rips filtration $(\text{VR}(\mathcal{X}; r),  T_r)_{r > 0}$ as a way to investigate the topology of $\mathcal{M}_n$, and the solutions of the associated MDP.

\bibliographystyle{plain}

\end{document}